\documentclass[12pt]{article}
\usepackage{geometry}
\usepackage[dvips]{graphicx}
\usepackage{float}
\usepackage{textcomp}
\usepackage{amsmath,amsthm,amssymb,latexsym}
\usepackage{subfigure}
\usepackage{rotating}
\usepackage[hidelinks]{hyperref}
\usepackage{cite}
\usepackage[nottoc]{tocbibind}
\usepackage{color}
\usepackage{multirow}
\usepackage{tabularx}
\usepackage{verbatim}
\usepackage{algorithm}
\usepackage{booktabs}
\newtheorem{theorem}{Theorem}[section]
\newtheorem{assumption}[theorem]{Assumption}
\newtheorem{corollary}[theorem]{Corollary}
\newtheorem{definition}[theorem]{Definition}
\newtheorem{proposition}[theorem]{Proposition}

\DeclareMathOperator*{\argmin}{argmin}
\DeclareMathOperator*{\argmax}{argmax}

\DeclareMathOperator{\interior}{int}

\newcommand{\set}[1]{\{#1\}}

\newcommand{\eq}[1]{\left\{\begin{array}{ll} #1 \end{array}\right}

\definecolor{electricgreen}{rgb}{0.0, 1.0, 0.0}

\begin{document}
\title{\textbf{An Inexact Restoration Direct Multisearch Filter Approach to Multiobjective Constrained Derivative-free Optimization}}
\author{Everton J. Silva\thanks{Ph.D Program in Mathematics, NOVA School of Science and Technology (NOVA FCT), Center for Mathematics and Applications (NOVA Math), Campus de Caparica, 2829-516 Caparica, Portugal
\texttt{(ejo.silva@campus.fct.unl.pt)}. Support for this author was provided by Fundação para a Ciência e a Tecnologia (Portuguese Foundation for Science and Technology) under the projects UI/BD/151246/2021, UIDP/00297/2020, and UIDB/00297/2020.} \and Ana Luísa Custódio \thanks{Department of Mathematics, NOVA School of Science and Technology (NOVA FCT), Center for Mathematics and Applications (NOVA Math), Campus de Caparica, 2829-516 Caparica, Portugal \texttt{(alcustodio@fct.unl.pt)}. Support for this author
was provided by Fundação para a Ciência e a Tecnologia (Portuguese Foundation for Science and Technology) under the projects UIDP/00297/2020 and UIDB/00297/2020.}}

\maketitle
\begin{abstract}
Direct Multisearch (DMS) is a well-established class of methods
for multiobjective derivative-free optimization, where constraints
are addressed by an extreme barrier approach, only evaluating
feasible points. In this work, we propose a filter approach,
combined with an inexact feasibility restoration step, to address
constraints in the DMS framework. The filter approach treats
feasibility as an additional component of the objective function,
avoiding the computation of penalty parameters or Lagrange
multipliers. The inexact restoration step attempts to generate new
feasible points, contributing to prioritize feasibility, a
requirement for the good performance of any filter approach.
Theoretical results are provided, analyzing the different types of
sequences of points generated by the new algorithm, and numerical
experiments on a set of nonlinearly constrained biobjective
problems are reported, stating the good algorithmic performance of
the proposed approach.
\\[5pt]
\end{abstract}

\section{Introduction}
\label{sec:int}
Multiobjective optimization is a challenging scientific domain,
not only from a theoretical point of view but also due to its
extensive range of
applications~\cite{brito2022,AHabiba_et_al_2022,FMoleiro_et_al_2021,YXing_et_al_2023}.
The presence of several conflicting objectives, that need to be
simultaneously optimized, changes the classical concept of problem
solution, which is no longer a single point. The goal is to
identify a set of nondominated points, meaning that it is
impossible to simultaneously improve all components of the
objective function for each one of these points.

This set of nondominated points is denominated as Pareto front and
constitutes the solution of the multiobjective optimization
problem. It can be supported by a convex, a nonconvex, a
connected, or a disconnected function, making its computation
difficult~\cite{miettinen1998,wiecek2016}. This task is even more
complicated if derivatives are not
available~\cite{audet2017,conn2009}. In this case, the problem to
be solved will be a multiobjective derivative-free optimization
problem.

Derivatives are an important tool in nonlinear
optimization~\cite{nocedal2006} since they can guide the search by
identifying directions of potential descent for the functions.
However, several reasons could explain their absence. For
instance, the problem by itself could be nonsmooth or, being
smooth, function evaluation could be computationally expensive,
preventing the use of numerical techniques to estimate
derivatives. Derivative-free optimization problems are often
associated to black-box functions, in the context of
simulation-based optimization~\cite{audet2017,conn2009}.

Multiobjective derivative-free optimization problems are commonly
addressed with heuristics, like is the case of evolutionary or
genetic algorithms~\cite{branke2008}. Although, if the cost
associated with computing a function value is high, these methods
are not an appropriate choice, often requiring a large number of
function evaluations. Additionally, theoretical analysis
supporting the numerical performance observed for these
heuristics, in general, is not yet available. In the last decade,
classes of derivative-free optimization methods have been
developed and analyzed for multiobjective optimization. In fact,
in a recent survey~\cite{larson2019}, multiobjective
derivative-free optimization was considered as ``an especially
open avenue of future research''.

According to~\cite{conn2009}, in single-objective derivative-free
optimization, three main classes of methods can be identified, for
which generalizations to multiobjective optimization, that compute
approximations to the complete Pareto front, have been proposed.
The first, directional direct search, was generalized to
multiobjective optimization with Direct
Multisearch~\cite{custodio2011}. This algorithmic class showed to
be competitive both in academic problems and in real applications,
even when compared with derivative-based multiobjective
optimization algorithms~\cite{andreani2022}.
BiMADS~\cite{audet2008}, MultiMADS~\cite{audet2010}, and
DMulti-MADS~\cite{bigeon2020} are other examples of algorithms of
directional direct search type.

The second class comprises trust-region methods based on quadratic
polynomial interpolation models~\cite{deshpande2016,ryu2014}.
These algorithms take advantage of the Taylor-like bounds that can
be established for the errors between these models and the
components of the objective function, proceeding by minimizing
each model by itself or using a scalarization
approach~\cite{eichfelder2008} to aggregate models in a
single-objective function to be minimized.

The final class is a generalization of implicit
filtering~\cite{CTKelley_2011} for bound constrained
multiobjective derivative-free optimization~\cite{cocchi2018}. At
each iteration of the proposed algorithm, a simplex Jacobian is
defined and used to compute an approximation to the multiobjective
steepest descent direction, which is then explored in a
line-search.

For any of the three above mentioned algorithmic classes,
convergence is analyzed, and numerical results support the ability
to compute good approximations to the complete Pareto front of a
given problem. Although, constraints have not been fully
addressed.

From the numerical point of view, the
works~\cite{audet2010,audet2008,cocchi2018,custodio2011,deshpande2016,ryu2014}
only report results on bound-constrained multiobjective
derivative-free optimization problems. Direct
Multisearch~\cite{custodio2011} and MultiMADS~\cite{audet2010} are
developed for general constraints. However, an extreme barrier
approach is adopted, only evaluating feasible points. To our
knowledge, DFMO~\cite{liuzzi2016} is the first algorithm that
explicitly addresses general constraints, using an exact penalty
approach.

In~\cite{bigeon2022}, two new constraint-handling strategies are
proposed for DMulti-MADS~\cite{bigeon2020}. The constraints are
aggregated into a single constraint violation function which is
used either in a two-phases approach (DMulti-MADS-TEB) or in a
progressive barrier approach (DMulti-MADS-PB). In DMulti-MADS-TEB,
the constraint violation function is minimized with the
MADS~\cite{audet2006} algorithm, in a single-objective setting,
until a feasible point is found. After, the algorithm proceeds as
in DMultiMads~\cite{bigeon2020}, making use of an extreme barrier
function, only evaluating feasible points. The second approach
(DMulti-MADS-PB) generalizes progressive
barrier~\cite{CAudet_JEDennis_2009} to multiobjective
optimization. The constraint violation function is considered as
an additional objective to be minimized, being rejected any trial
point with constraint violation function value above a given
threshold. This threshold is progressively decreased along the
iterations. Each iteration explores two poll centers,
corresponding to feasible and infeasible points, respectively.
Progressive barrier~\cite{CAudet_JEDennis_2009} can be regarded as
an evolution of filter methods~\cite{audet2004}, initially
proposed by Fletcher and Leyffer~\cite{fletcher2002} for
single-objective derivative-based optimization, as an alternative
to address general constraints. The work~\cite{fletcher2007}
provides a survey on filter methods.

In this work, we propose the integration of a filter approach in
DMS, to address multiobjective derivative-free optimization
problems with general constraints. Similarly to~\cite{bigeon2022},
the violations of the constraints are aggregated in an extra
objective function component to be minimized. However, differently
from~\cite{bigeon2022}, a proper criterion is defined to decide
when to explore feasible or infeasible points at a given iteration
(but never both) and the maximum value allowed for constraint
violation is never updated. Instead, when the point to be explored
at a given iteration is infeasible, the algorithm makes use of an
inexact feasibility restoration step. The motivation behind it is
that one should not evaluate the possibly expensive objective
function, without first trying to restore (or at least improve)
feasibility. Inexact restoration methods are well suited for this
purpose and were already explored in several works in
single-objective derivative-free optimization
(see~\cite{pilotta2015,Birgin2022,LFBueno_et_al_2013,Echebest2017,Ferreira2017,martinez_sobral2013}).
For a survey on general inexact restoration feasibility
approaches, see~\cite{pillota2005}.

The paper is organized as follows. Section~\ref{sec:algorithm}
describes the proposed algorithmic structure. The theoretical
properties of the sequences of points generated by the algorithm
are established in Section~\ref{sec:convergence}.
Section~\ref{sec:implementation} provides some details respecting
the numerical implementation used to compute the results reported
in Section~\ref{sec:numerical_experiments}. The paper ends in
Section~\ref{sec:conclusions} with some final remarks.

\section{An inexact restoration DMS filter algorithm}
\label{sec:algorithm}

Let us consider the multiobjective optimization problem with
general constraints (linear and nonlinear), defined by:
\begin{equation}
\begin{array}{l}
\displaystyle \min_{x\in X}\  F(x)=(f_{1}(x),\ldots ,f_{m}(x))^\top \\
\text{s.t. } C(x)\leq 0,\\
\end{array}
\label{MOO}
\end{equation}
where
$F:X\subseteq\mathbb{R}^{n}\to\set{\mathbb{R}\cup\{+\infty\}}^m$,
with $m\in \mathbb{N},m\geq 2$,
$C:X\subseteq\mathbb{R}^n\to\set{\mathbb{R}\cup\set{+\infty}}^p$,
$p\in\mathbb{N}$, and $X\subseteq\mathbb{R}^n$ denotes the set of
unrelaxable constraints~\cite{conn2009}. Therefore, the feasible
region, $\Upsilon$, of the multiobjective problem, assumed to be
nonempty, is given by $\Upsilon =X\cap \Omega\neq \emptyset$,
where $\Omega=\set{x\in \mathbb{R}^{n}\mid C(x)\leq 0}$.

In multiobjective optimization, the concept of Pareto dominance is
essential for point comparison. To describe it, we will make use
of the strict partial order induced by the cone
$\mathbb{R}^m_+=\set{z\in\mathbb{R}^m\mid z\geq 0},$ defined as:
$$F(x)\prec_F F(y) \Leftrightarrow
F(y)-F(x)\in\mathbb{R}_+^m\backslash \set{0}.$$ Given two points
$x,y$ in $\Upsilon$, we say that $x\prec_F y$, i.e., $x$ dominates
$y$, when $F(x)\prec_F F(y)$.

We are now in conditions of characterizing efficient points of
Problem~(\ref{MOO}).
\begin{definition}
A point $\bar{x}\in\Upsilon$ is said to be a global efficient
point of Problem~(\ref{MOO}) if there is no $y\in \Upsilon$ such
that $y\prec_F \bar{x}$. If there exists a neighborhood
$N(\bar{x})$ of $\bar{x}$ such that the previous property holds in
$\Upsilon \cap N(\bar{x})$, then $\bar{x}$ is called a local
efficient point of Problem~(\ref{MOO}).
\end{definition}

The image of the set of global efficient points for
Problem~(\ref{MOO}) constitutes the solution of the multiobjective
optimization problem and is denoted by the Pareto front.

In applications, unrelaxable constraints are often associated to
physical conditions that can not be violated (otherwise, it will
be impossible to evaluate the objective function). Thus, our
approach will address them with an extreme barrier function, only
evaluating points that satisfy these constraints. In the problem
definition, function $F$ will be replaced by $F_X$, defined as:
\begin{equation}
F_X(x)= \eq{ F(x)&, \text{ if } x\in X\\
(+\infty,+\infty,\ldots,+\infty)^\top&, \text{ otherwise.}}.
\label{barrier_approach}
\end{equation}

Although, points do not always need to remain feasible regarding
the relaxable constraints, defined by function $C(\cdot)$. We
intend to minimize this violation and a maximum threshold, $h_{\tt
max}>0$, will be allowed for it. For that, following the approach
of~\cite{audet2004}, proposed for single-objective optimization,
we will consider an additional nonnegative objective function
component, $h$, corresponding to an aggregated violation of the
relaxable constraints. Function $h$ should satisfy $h(x)=0$ if and
only if $x\in\Omega$. A possibility for its definition could be:
\begin{equation}
\label{constraintv_function} h(x)=\|C(x)_+\|^r,
\end{equation}
where $\|\cdot\|$ is a vector norm, $r>0$, and $C(x)_+$ is the
vector of $p$ constraint values, defined for $i=1,\ldots,p$ by
\begin{equation}
c_i(x)_+=\eq{  c_i(x) &, \text{ if } c_i(x)>0\\
0 &, \text{ otherwise }}.
\end{equation}
As an example, considering the $\ell_2$-norm and $r=2$, we have:
\begin{equation*}
h(x)=\|C(x)_+\|_2^2=\displaystyle\sum_{i=1}^{p}\max\set{0,c_i(x)}^2.
\end{equation*}

An approximation to the solution of Problem~(\ref{MOO}) will be
computed by solving
\begin{equation}
\min_{x\in X}\  \bar{F}_{\bar{X}} (x),\label{MOO2}
\end{equation}
where $\bar{F}(x)=(f_{1}(x),\ldots ,f_{m}(x),h(x))^\top$ and
$\bar{X}=\{x\in X \mid h(x)\leq h_{\tt max}\}$.

The algorithm proposed to address Problem~(\ref{MOO2}) is
developed under the DMS framework. Thus, a list of feasible
nondominated points, regarding the unrelaxable constraints, and
corresponding step sizes is considered. In a simplified way, each
iteration tries to improve this list of points, by adding new
nondominated points to it and removing dominated ones. The
procedure follows Algorithm 2.1 in~\cite{custodio2011}. At the end
of the optimization process, the points in the list that satisfy
$h(x)=0$ constitute the approximation to the Pareto front of the
original problem.

Each iteration starts with the selection of an iterate point (and
corresponding step size) from the list. This iterate point is
always feasible regarding the constraints defining the set $X$,
but can be infeasible with respect to the constraints defining the
set $\Omega$. In Section~\ref{poll_center_selection}, we will
propose rules for the selection of iterate points. After this
selection, a search step and possibly a poll step are performed.

The search step is optional, not required for establishing
convergence, so we will focus on the poll step. At this step, a
local search is performed around the current iterate, by exploring
directions belonging to a positive spanning set, scaled by the
step size. Details on the properties that this set of directions
needs to satisfy will be provided in
Section~\ref{sec:convergence}.

Due to the presence of relaxable constraints, the iterate point
could be infeasible regarding the set $\Omega$ (and the original
problem). In this situation, considering the expensive nature of
the objective function, it would be wiser to try to restore
feasibility, before initiating the poll procedure. For that, the
following inexact feasibility restoration problem will be solved:
\begin{equation*}
\begin{array}{l}
\displaystyle \min_{y\in X}\  \frac{1}{2}\|y-x_k\|^2\\
\text{s.t. } h(y)\leq \xi(\alpha_k) h(x_k),\\
\end{array}\label{IRP}
\end{equation*}
where $x_k$ and $\alpha_k$ denote the current iterate and
associated step size, respectively. Function $\xi:(0,+\infty)\to
]0,1[$ is continuous and satisfies $\xi(t)\to 0$ when $t\downarrow
0$. Since the feasible region $\Upsilon$ is nonempty, this problem
is well-defined. Solving the inexact restoration problem, before
polling being attempted, is an explicit way of prioritizing
feasibility. The definition of the $\xi(\cdot)$ function ensures
that if the stepsize goes to zero, in general, meaning that a
limit point is being attained, then feasibility regarding $\Omega$
is also being restored.

The list of points is a dynamic set, that will allow to classify
iterations as being successful or unsuccessful. Similarly to the
original implementation of DMS~\cite{custodio2011}, an iteration
is said to be successful if the iterate list changes, meaning that
at least one new feasible nondominated point was added to it. Here
feasibility respects only to the set $X$, of unrelaxable
constraints, and dominance function $\bar{F}$. Unsuccessful
iterations keep the list unchanged. Differently
from~\cite{bigeon2022}, comparisons are made for all the points in
the list, regardless of the associated feasibility.

The rule for updating the stepsize parameter follows what is
classical in directional direct search. Therefore, for successful
iterations, the step size parameter is either increased or kept
constant, i.e., $\alpha_{k,\tt new}\in [\alpha_k,\gamma \alpha_k]$
for $\gamma\geq 1$, for all the feasible nondominated points added
and for the poll center, if it remains in the list. At
unsuccessful iterations, the step size of the poll center is
decreased, i.e., $\alpha_{k,\tt new}\in
[\beta_1\alpha_k,\beta_2\alpha_k]$ for $0<\beta_1\leq \beta_2<1$.

Algorithm~\ref{DMS-FILTER-IR} details the DMS-FILTER-IR
multiobjective derivative-free constrained optimization method.

\begin{algorithm}
\begin{rm}
\begin{description}
{\small \vspace{1ex}
\item[Initialization] \ \\
Choose an initial step size parameter $\alpha_0 > 0$, $0 <
\beta_1\leq \beta_2 < 1$, and  $\gamma \geq 1$. Let $\mathcal{D}$
be a (possibly infinite) set of positive spanning sets, with
directions $d$ satisfying $0<d_{min}\leq\|d\|\leq d_{max}$. Define
$h(\cdot)$, the nonnegative violation aggregation function, and
$h_{\tt max}>0$, the maximum violation allowed for it. Define
$\xi(\cdot)$, to be used in the inexact restoration step. Consider
$x_0\in X$ such that $h(x_0)\leq h_{\tt max}$ and set
$L_0=\set{(x_0;\alpha_0)}$.\vspace{1ex}

\item[For $k=0,1,2,\ldots$] \ \\
\begin{enumerate}
\item[1.] {\bf Selection of an iterate point:} Order the list
$L_k$ according to some criteria and select $(x_k;\alpha_k)\in
L_k$ as the current iterate and step size parameter.\ \\

\item[2.] {\bf Search step:} Compute a finite set of points
$\{z_s\}_{s\in S}$ (in a mesh if $\bar\rho(\cdot)=0$, see
Section~\ref{sec:convergence}). Evaluate~$\bar{F}_{\bar{X}}$ at
each element of $\{z_s\}_{s\in S}$. Use $L_{add}=\{(z_s;\alpha_k),
s\in S\}$ to generate $L_{trial}$, by updating $L_k$ with the new
nondominated points in $L_{add}$ and removing the dominated ones.
If $L_{trial} \neq L_k$, then declare the iteration (and the
search step) successful,
set $L_{k+1}=L_{trial}$, and go to Step 5.\ \\

\item[3.] {\bf Inexact Restoration step:} If $h(x_k)>0$ then
define and solve the inexact restoration problem $y^*\in
\argmin_{y\in X}\  \frac{1}{2}\|y-x_k\|^2$ subject to $h(y)\leq
\xi(\alpha_k) h(x_k)$ (in a mesh if $\bar\rho(\cdot)=0$, see
Section~\ref{sec:convergence}). Evaluate $\bar{F}_{\bar{X}}$ at
$y^*$, define $L_{add}=\{(y^*;\alpha_k)\}$ to generate
$L_{trial}$, by updating $L_k$ with the new nondominated point in
$L_{add}$ and removing the dominated ones. If $L_{trial} \neq
L_k$, then declare the iteration (and the inexact restoration
step) as successful, set
$L_{k+1}=L_{trial}$, and go to Step 5.\ \\

\item[4.] {\bf Poll step:} Choose a positive spanning set $D_k$
from the set $\mathcal{D}$. Evaluate $\bar{F}_{\bar{X}}$ at
$P_k=\{x_k+\alpha_k d\mid d\in D_k\}$, define
$L_{add}=\{(x_k+\alpha_k d;\alpha_k), d\in D_k \}$ to generate
$L_{trial}$, by updating $L_k$ with the new nondominated points in
$L_{add}$ and removing the dominated ones. If $L_{trial} \neq
L_k$, then declare the iteration (and the poll step) as successful
and set $L_{k+1}=L_{trial}$. Otherwise, declare the iteration as
unsuccessful and set $L_{k+1}=L_k$.\ \\

\item[5.] {\bf Step size parameter update:} If the iteration was
successful, then maintain or increase the corresponding step size
parameter, by considering $\alpha_{k,{\tt new}}\in
[\alpha_k,\gamma\alpha_k]$. Replace all the new points $(x_k
+\alpha_k d; \alpha_k)$ in $L_{k+1}$ by $(x_k+ \alpha_k d;
\alpha_{k,{\tt new}})$, when success is coming from the poll step,
or $(y^*;\alpha_k)$ in $L_{k+1}$ by $(y^*;\alpha_{k,{\tt new}})$,
when success is coming from the inexact restoration step, or
$(z_s;\alpha_k)$ in $L_{k+1}$ by $(z_s;\alpha_{k,{\tt new}})$,
when success is coming from the search step. Replace also
$(x_k;\alpha_k)$, if in $L_{k+1}$, by $(x_k;\alpha_{k,{\tt
new}})$. Otherwise, decrease the step size parameter, by choosing
$\alpha_{k,{\tt new}}\in [\beta_1\alpha_k,\beta_2\alpha_k]$, and
replace the poll pair $(x_k;\alpha_k)$ in $L_{k+1}$ by
$(x_k;\alpha_{k,{\tt new}})$. \vspace{1ex}
\end{enumerate}
\item[EndFor]\ \\}
\end{description}
\end{rm}
\label{DMS-FILTER-IR} \caption{DMS-FILTER-IR method for
constrained MOO}
\end{algorithm}


\section{Convergence analysis}
\label{sec:convergence} Following the reasoning
of~\cite{audet2004,custodio2011}, this section analyzes the
properties of the different sequences of points generated by
DMS-FILTER-IR.

\subsection{Globalization strategies}

In classical directional direct search, the first step in the
convergence analysis is globalization, i.e., to ensure the
existence of a subsequence of step size parameters that converges
to zero. Two different strategies can be adopted. The first,
analyzed next, requires that all points generated by the algorithm
lie in an implicit mesh, corresponding to an integer lattice.

For establishing the result, we will need the following
assumption.

\begin{assumption}
\label{assumption1} The set $\mathcal{S}:=\bigcup_{j=1}^m\set{x\in
X \mid h(x)\leq h_{\tt max}\wedge f_j(x)\leq f_j(x_0)}$ is a
nonempty compact set.
\end{assumption}

In~\cite{custodio2011}, when DMS was originally proposed, the
directions to be used by the algorithm at each iteration belonged
to a positive spanning set $D_k$, selected from $\mathcal{D}$,
whose directions are built as nonnegative integer combinations of
the columns of a set $D$. The following assumption formalizes the
conditions imposed on the set $D$, in order to satisfy the
integrality requirements.

\begin{assumption}
\label{assumption2} The set $D$ of positive spanning sets is
finite and the elements of $D$ are of the form $G\bar{z}_j$,
$j=1,\ldots,|D|$,  where $G\in\mathbb{R}^{n\times n}$ is a
nonsingular matrix and each $\bar{z}_j$ is a vector in
$\mathbb{Z}^n$.
\end{assumption}

In the presence of general constraints, and possibly nonsmooth
functions, it is required to consider an infinite set of
directions $\mathcal{D}$, which should be dense (after
normalization) in the unit sphere~\cite{audet2006}. The set $D$ is
used for building the directions in $\mathcal{D}$.

\begin{assumption}
\label{assumption3} Let $D$ represent a positive spanning set
satisfying Assumption~\ref{assumption2}, with elements $d_k\in
D_k\in\mathcal{D}$ obtained as nonnegative integer combinations of
the columns of $D$.
\end{assumption}

To comply with the integrality requirements, additional conditions
need to be imposed in the update of the step size parameter.

\begin{assumption}
\label{assumption4} Let $\tau>1$ be a rational number and $m^{\tt
max}\geq 0$ and $m^{\tt min}\leq -1$ integers. If the iteration is
successful, then the step size parameter is maintained or
increased by considering $\alpha_{k,{\tt new}}=\tau^{m^+}\alpha_k$
with $m^+\in\set{0,\ldots,m^{\tt max}}$. If the iteration is
unsuccessful, then the step size parameter is decreased by setting
$\alpha_{k,{\tt new}}=\tau^{m^-}\alpha_k$, with $m^-\in\set{m^{\tt
min},\ldots,-1}$.
\end{assumption}

The update rule of Algorithm~1 complies with the one of
Assumption~\ref{assumption4} by setting $\beta_1=\tau^{m^{\tt
min}}$, $\beta_2=\tau^{-1}$, and $\gamma=\tau^{m^{\tt max}}$.

In addition, the points generated both at the search and at the
inexact restoration step need to lie in the implicit mesh
considered at each iteration by the algorithm.

\begin{assumption}
\label{assumption5} At iteration $k$, the search and the inexact
restoration steps in Algorithm~1 only evaluate points in
$$M_k=\displaystyle\bigcup_{x\in E_k} \set{x+\alpha_k Dz\mid z\in
\mathbb{N}^{|D|}_0},$$ where $E_k$ represents the set of all
points evaluated by the algorithm previously to iteration $k$.
\end{assumption}

The following theorem states that there is at least one
subsequence of iterations for which the step size parameter
converges to zero. The proof is omitted since it uses exactly the
same arguments of Theorem A.1 in~\cite{custodio2011}.

\begin{theorem}
\label{teoalfa_integerlattices} Let Assumption~\ref{assumption1}
hold. Under one of the Assumptions~\ref{assumption2}
or~\ref{assumption3} combined with Assumptions~\ref{assumption4}
and~\ref{assumption5}, DMS-FILTER-IR generates a sequence of
iterates satisfying
\begin{equation*}
\displaystyle \liminf_{k\to+\infty} \alpha_k =0.
\end{equation*}
\end{theorem}

Globalization can also be ensured by requiring sufficient decrease
to accept new points, by means of a forcing function. A forcing
function $\rho:(0,+\infty)\rightarrow (0,+\infty)$ is a continuous
and nondecreasing function, that satisfies $\rho(t)/t \rightarrow
0$ when $t \downarrow 0$ (see~\cite{kolda2003}). Typical examples
of forcing functions are $\rho(t) = \eta_1t^{1+\eta_2}$, for
$\eta_1,\eta_2
> 0$. Definition~\ref{new_dominance} traduces the new dominance
relationship considered.

\begin{definition}\label{new_dominance}
Let $y$ belong to $\bar{X}$ and $L$ be a list of nondominated
points in $\bar{X}$. We say that $y$ is dominated if:
$$\exists x\in L: \bar{F}(x)-\rho(\alpha)\leq \bar{F}(y),$$
where $\rho(\cdot)$ denotes a forcing function and $\alpha$ the
step size associated to the current iteration.
\end{definition}

Figure~\ref{suf_decre} illustrates the situation for the list of
infeasible points, $L$, whose images by function $\bar{F}$
correspond to the green dots. $D(L)\subset\mathbb{R}^{m+1}$
represents the image of the set of points dominated by the points
in $L$ and $D(L; \rho(\alpha))$ denotes the set of points whose
distance in the ${\ell}_{\infty}$ norm to $D(L)$ is no larger than
$\rho(\alpha)>0$. Points will be accepted if their image by
$\bar{F}$ does not belong to $D(L; \rho(\alpha))$, ensuring an
increase in the hypervolume associated to the list of points of at
least $(\rho(\alpha))^{m+1}$ (see Lemma 3.1
in~\cite{custodio2021}).

\begin{figure}[H]
\centering
\includegraphics[scale=0.5]{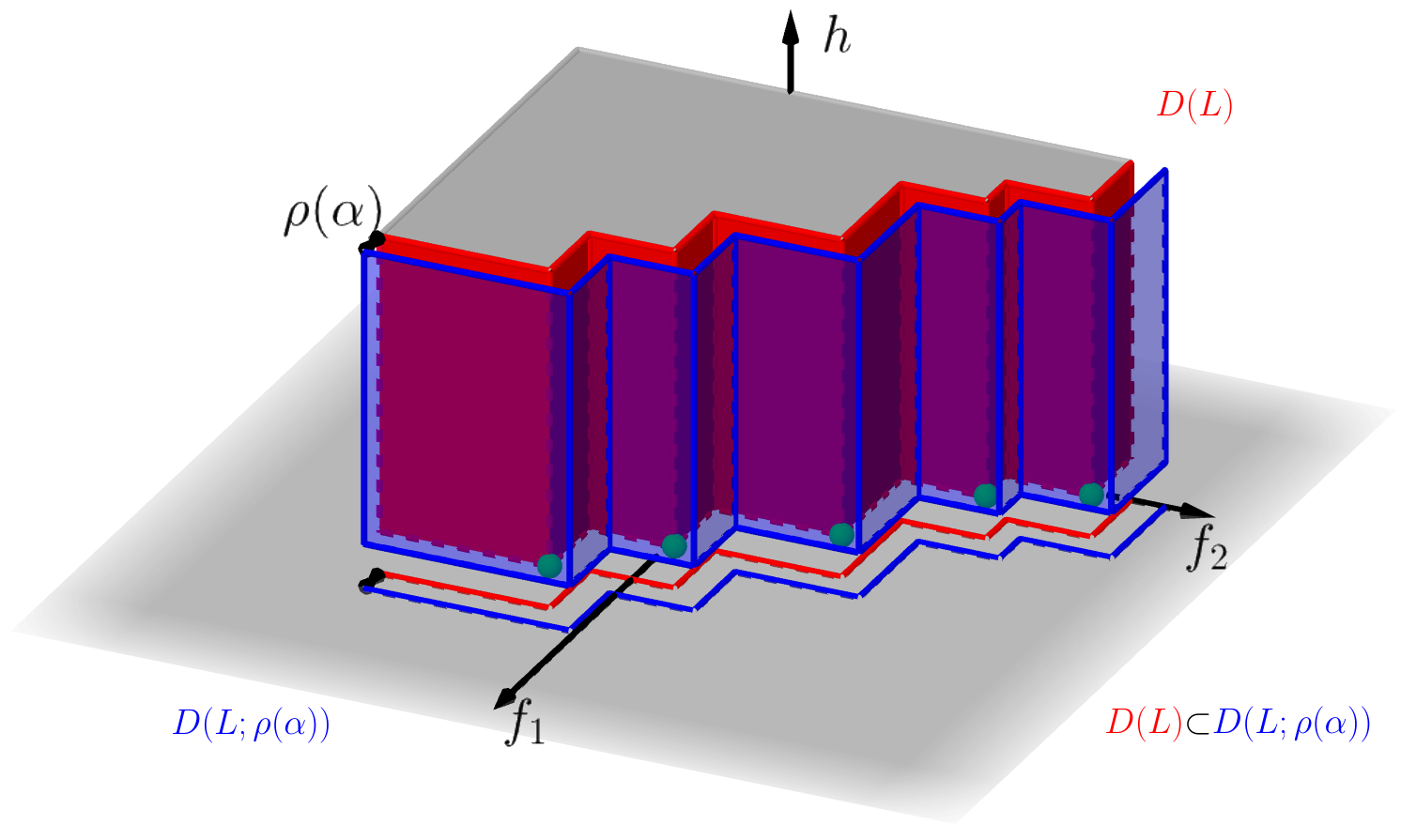}
\caption{Globalization strategy based on a sufficient decrease
condition in Algorithm~1.} \label{suf_decre}
\end{figure}

The assumptions required to ensure globalization under a
sufficient decrease approach are slightly different.

\begin{assumption}
\label{assumption6} The function $F:\mathbb{R}^n\to
\mathbb{R}^{m}$ is bounded in the set
$\mathcal{S}:=\bigcup_{j=1}^m\set{x\in X \mid h(x)\leq h_{\tt
max}\wedge f_j(x)\leq f_j(x_0)}$.
\end{assumption}

By definition, $0\leq h(x)\leq h_{\tt max}$. Thus,
Assumption~\ref{assumption6} guarantees that the function
$\bar{F}:\mathbb{R}^n\to \mathbb{R}^{m+1}$, defined by
$\bar{F}=(F,h)$, is also bounded in $\mathcal{S}$.

The following theorem states the convergence to zero of at least
one subsequence of step sizes, when using a globalization strategy
based on sufficient decrease.

\begin{theorem}
\label{teoalfa_sufficientdecrease} Consider a globalization
strategy based on imposing a sufficient decrease condition and let
Assumption~\ref{assumption6} hold. Then, DMS-FILTER-IR generates a
sequence of iterates satisfying
\begin{equation*}
\displaystyle\liminf_{k\to +\infty} \alpha_k =0.
\end{equation*}
\end{theorem}

\begin{proof}
Let us assume that $\displaystyle\liminf_{k\to +\infty}
\alpha_k\neq 0$, meaning that there is $\alpha^*$ such that
$\alpha_k>\alpha^*$ for all $k$. At each unsuccessful iteration
$k$, the corresponding step size parameter is reduced by at least
$\beta_2\in (0,1)$. Thus, the number of successful iterations must
be infinite.

Successful iterations increase the hypervolume of the dominated
region associated to the function $\bar{F}$ and the list of points
in at least $(\rho(\alpha_k))^{m+1}$, where $\alpha_k$ represents
the step size associated with the current iteration (see Lemma 3.1
in~\cite{custodio2021}).

Since $\rho(\cdot)$ is  a nondecreasing function, which satisfies
$\rho(t)>0$ for $t>0$, there exists $\rho^*>0$ such that
$\rho(\alpha_k)\geq \rho(\alpha^*)=\rho^*$. Thus, any successful
iteration will increase the hypervolume of the dominated region
associated to the list of points for function $\bar{F}$ in at
least $(\rho^*)^{m+1}$, contradicting
Assumption~\ref{assumption6}.
\end{proof}

\subsection{Sequences and stationarity}

To establish the convergence of direct search methods of
directional type, the behavior of the algorithms needs to be
analyzed at limit points of particular sequences of unsuccessful
iterates, denoted by refining subsequences.

\begin{definition}
A subsequence $\set{x_k}_{k\in K}$ of iterates generated by
DMS-FILTER-IR, corresponding to unsuccessful poll steps, is said
to be a refining subsequence if $\set{\alpha_k}_{k\in K}$
converges to zero.
\end{definition}

Assumption~\ref{assumption1},
Theorems~\ref{teoalfa_integerlattices}
or~\ref{teoalfa_sufficientdecrease}, and the updating strategy of
the step size allow to establish the existence of at least one
convergent refining subsequence (see, e.g.,~\cite[Section
7.3]{conn2009}). The limit point $\bar{x}$ of a refining
subsequence $\set{x_k}_{k\in K}$ is said to be a refined point. As
suggested by the numerical experiments reported in
Section~\ref{sec:numerical_experiments}, it is common that the
algorithm will generate several refined points.
In~\cite{bigeon2020} the same type of result is established using
the concept of linked sequences.

Refined points, corresponding to limit points of sequences of
unsuccessful iterates, will be the natural candidates to
Pareto-Clarke stationarity. For defining it, in a nonsmooth
setting, we will need a generalization of the tangent cone
commonly used in nonlinear programming, namely the Clarke's
tangent cone~\cite{clarke1990}.

\begin{definition}
\label{tangent_cone} A vector $d\in\mathbb{R}^n$ is said to be a
Clarke tangent vector to the set $Y\subseteq\mathbb{R}^n$ at the
point $\bar{x} \in cl(Y)$ if for every sequence $\set{y_k}$ of
elements of $Y$ that converges to $\bar{x}$ and for every sequence
of positive real numbers $\set{t_k}$ converging to zero, there
exists a sequence of vectors $\set{w_k}$ converging to $d$ such
that $y_k+t_k w_k\in Y$.
\end{definition}

The set of all Clarke tangent vectors to $Y$ at $\bar{x}$ is
called the Clarke tangent cone to $Y$ at $\bar{x}$, and is denoted
by $T^{Cl}_Y(\bar{x})$. The tangent cone is the closure of another
relevant cone for the following analysis, namely the Clarke's
hypertangent cone~\cite{clarke1990}.

\begin{definition}
A vector $d\in\mathbb{R}^n$ is said to be a Clarke hypertangent
vector to the set $Y\subseteq\mathbb{R}^n$ at the point $\bar{x}
\in Y$ if there exists a scalar $\epsilon>0$ such that $y+t w \in
Y$ for all $y\in Y\cap B(\bar{x};\epsilon)$, $w\in B(d;\epsilon)$,
and $0<t<\epsilon$.
\end{definition}

The set of all hypertangent vectors to $Y$ at $\bar{x}$ is called
the hypertangent cone to $Y$ at $\bar{x}$ and is denoted by
$H^{Cl}_Y(\bar{x})$. Whenever the interior of $T^{Cl}_Y(\bar{x})$
is nonempty, $H^{Cl}_Y(\bar{x}) = \interior(T^{Cl}_Y(\bar{x}))$.

The notion of directional derivative needs also to be generalized
to nonsmooth functions, accounting for the presence of
constraints~\cite{clarke1990,JJahn_1996}.

\begin{definition}
Let $g:\mathbb{R}^n\to\mathbb{R}$ be Lipschitz continuous near
$\bar{x}\in Y$. The Clarke-Jahn generalized derivative of $g$
along $d\in H^{Cl}_Y(\bar{x})$ is defined as:
\begin{equation*}
g^\circ(\bar{x};d):= \limsup_{x \to \bar{x}, x \in Y \atop t
\downarrow 0, x+t d \in Y} \dfrac{g(x+t d)-g(x)}{t}.
\end{equation*}
\end{definition}

When $H^{Cl}_Y(\bar{x})$ is nonempty, the Clarke-Jahn generalized
directional derivatives along directions $v$ in
$T^{Cl}_Y(\bar{x})$ can be computed by taking limits of sequences
of directions belonging to the hypertangent cone~\cite{audet2006}.

\begin{proposition}
Let $g:\mathbb{R}^n\to\mathbb{R}$ be Lipschitz continuous near
$\bar{x}\in Y$ and assume that $H^{Cl}_Y(\bar{x})$ is nonempty.
The Clarke-Jahn generalized derivative of $g$ along $v\in
T^{Cl}_Y(\bar{x})$ can be computed as:
\begin{equation*}
g^\circ(\bar{x};v):=\lim_{d\to v \atop d\in H^{Cl}_Y(\bar{x})}
g^\circ(\bar{x};d).
\end{equation*}
\end{proposition}

We are now in conditions of establishing what would be a
first-order stationarity result for Problem~(\ref{MOO}). The
following definition states essentially that there is no direction
in the tangent cone that is descent for all components of the
objective function.
\begin{definition}
Let $F$ be Lipschitz continuous near a point $\bar{x}\in
\Upsilon$. We say that $\bar{x}$ is a Pareto-Clarke critical point
of $F$ in $\Upsilon$, if for each direction $d\in
T^{Cl}_{\Upsilon}(\bar{x})$, there exists a
$j=j(d)\in\set{1,\ldots,m}$ such that
$f_{j}^{\circ}(\bar{x};d)\geq 0$.
\end{definition}

If the objective function is differentiable, the previous
definition can be reformulated using the columns of the Jacobian
matrix.

\begin{definition}
Let $F$ be strictly differentiable at a point $\bar{x}\in
\Upsilon$. We say that $\bar{x}$ is a Pareto-Clarke-KKT critical
point of $F$ in $\Upsilon$, if for each direction $d\in
T_{\Upsilon}^{Cl}(\bar{x})$, there exists a
$j=j(d)\in\set{1,\ldots,m}$ such that $\nabla f_{j}(x)^\top d \geq
0$.
\end{definition}

Finally, for establishing the desired stationarity results we need
the definition of refining directions associated with refined
points.

\begin{definition}
\label{refdirection} Let $\bar{x}$ be the limit point of a
convergent refining subsequence $\set{x_k}_{k\in K}$. If the limit
$\displaystyle \lim_{k\in K^\prime} \frac{d_k}{\|d_k\|}$ exists,
where $K^\prime\subseteq K$ and $d_k\in D_k$, and if $x_k+\alpha_k
d_k$ is feasible, for sufficiently large $k\in K^\prime$, then
this limit is said to be a refining direction for $\bar{x}$.
\end{definition}

The convergence analysis will initiate with the study of the
behavior of the algorithm along refined directions, belonging to
the hypertangent cone, both computed at a refined point.

\subsection{DMS-FILTER-IR convergence results} \label{op_cond}
In this section, we present the main convergence results of
DMS-FILTER-IR. The analysis is based on the property established
by Proposition~\ref{prop_dominance}, valid for any of the two
globalization strategies considered. Function $\bar{\rho}(\cdot)$
corresponds to the forcing function $\rho(\cdot)$, when
globalization is based on a sufficient decrease condition, or is
defined as the null function ($\bar{\rho}(\cdot)=0$), when
globalization results from the use of integer lattices. For
simplicity when stating the results, we also define $f_{m+1}$ as
being equal to $h$, the aggregated violation of the relaxed
constraints.

\begin{proposition}\label{prop_dominance}
Let $x\in L$ and $y\in \bar{X}$ be a dominated point at an
iteration associated with the step size $\alpha$. Then:
$$\exists j\in\{1,\ldots,m+1\}:f_j(y)>f_j(x)-\bar{\rho}(\alpha).$$
\end{proposition}
\begin{proof}
Point $y$ is dominated, meaning that there is $z\in L$ such that
$$\bar{F}(z)-\bar{\rho}(\alpha)\leq \bar{F}(y),$$
and, if $\bar{\rho}(\cdot)=0$, $\bar{F}(z)\neq\bar{F}(y)$.

Suppose $\bar{F}(x)-\bar{\rho}(\alpha)\geq \bar{F}(y)$. Then,
$$\bar{F}(x)\geq \bar{F}(x)-\bar{\rho}(\alpha)\geq \bar{F}(z)-\bar{\rho}(\alpha).$$
Moreover, if $\bar{\rho}(\cdot)=0$ then $\bar{F}(z)\neq
\bar{F}(x)$. Thus, $x$ would be dominated by $z$, which is not
possible, since both $x$ and $z$ belong to the list of
nondominated points $L$.
\end{proof}

As the step size approaches zero, the poll step will allow to
recover the local sensitivities of the objective function.
Together with some proper smoothness assumptions, we can establish
that there is no locally improving direction, for the adequate
problem. The proof follows directly from Theorem 4.8
in~\cite{custodio2011}. For completeness, we reproduce it in
Theorem~\ref{teo_opt_f}, with the due adaptations.
\begin{theorem}
\label{teo_opt_f} Consider $\{x_k\}_{k\in K}$, a refining
subsequence generated by DMS-FILTER-IR, converging to the refined
point $\bar{x}\in \bar{X}$. Assume that $\bar{F}$ is Lipschitz
continuous near $\bar{x}$. Let $d\in H^{Cl}_{\bar{X}}(\bar{x})$ be
a refining direction for $\bar{x}$, associated with the refining
subsequence $\{x_k\}_{k\in K}$. Then,
$$\exists j\in\set{1,\ldots,m+1}: f_j^\circ(\bar{x};d)\geq
0.$$
\end{theorem}
\begin{proof}
Since $\{x_k\}_{k\in K}$ is a refining subsequence converging to
$\bar{x}$, we have $\lim_{k\in K}x_k=\bar{x}$, $\lim_{k\in
K}\alpha_k=0$ and $k\in K$ is the index of an unsuccessful
iteration.

Consider $K'\subseteq K$ such that $\lim_{k\in
K'}\frac{d_k}{\|d_k\|}=d\in H_{\bar{X}}^{Cl}(\bar{x})$, with $d_k$
a poll direction used at iteration $k$. Let
$j\in\{1,\ldots,m+1\}$. Then,

\begin{equation*}
f_{j}^\circ (\bar{x};d) =\limsup_{x\rightarrow \bar{x}, x \in
\bar{X} \atop t \downarrow 0, x+t d \in \bar{X}} \dfrac{f_{j}(x+t
d)-f_{j}(x)}{t}\geq\limsup_{k\in K'}
\dfrac{f_j(x_k+\alpha_k\|d_k\|d)-f_j(x_k)}{\alpha_k\|d_k\|}=
\end{equation*}
\begin{equation*}
=\limsup_{k\in K'}\left( \dfrac{f_j(x_k+\alpha_k
d_k)-f_j(x_k)+\bar{\rho}(\alpha_k)}{\alpha_k\|d_k\|}+
\dfrac{f_j(x_k+\alpha_k \|d_k\|d)-f_j(x_k+\alpha_k
d_k)}{\alpha_k\|d_k\|}-\frac{\bar{\rho}(\alpha_k)}{\alpha_k\|d_k\|}\right)\geq
\end{equation*}
\begin{equation*}
\geq\limsup_{k\in K'} \dfrac{f_j(x_k+\alpha_k
d_k)-f_j(x_k)+\bar{\rho}(\alpha_k)}{\alpha_k\|d_k\|}+\liminf_{k\in
K'}\left( \dfrac{f_j(x_k+\alpha_k \|d_k\|d)-f_j(x_k+\alpha_k
d_k)}{\alpha_k\|d_k\|}-\frac{\bar{\rho}(\alpha_k)}{\alpha_k\|d_k\|}\right)
\end{equation*}
Since each $d_k$ is lower bounded by $d_{min}>0$, the definition
of $\bar{\rho}(\cdot)$ and the properties of the forcing function,
allow to establish $\lim_{k\in
K'}\frac{\bar{\rho}(\alpha_k)}{\alpha_k\|d_k\|}=0$.

Moreover, the Lipschitz continuity of $\bar{F}$ ensures that
$$\left|\dfrac{f_j(x_k+\alpha_k \|d_k\|d)-f_j(x_k+\alpha_k
d_k)}{\alpha_k\|d_k\|}\right|\leq
L\left\|d-\frac{d_k}{\|d_k\|}\right\|,$$ where $L$ represents the
maximum of the Lipschitz constants associated to each one of the
objective function components. Thus, the fact that $\lim_{k\in
K'}\frac{d_k}{\|d_k\|}=d$ ensures that
$$\lim_{k\in K'}\dfrac{f_j(x_k+\alpha_k \|d_k\|d)-f_j(x_k+\alpha_k
d_k)}{\alpha_k\|d_k\|}=0.$$

We then have,
$$f_{j}^\circ (\bar{x};d)\geq \limsup_{k\in K'} \dfrac{f_j(x_k+\alpha_k
d_k)-f_j(x_k)+\bar{\rho}(\alpha_k)}{\alpha_k\|d_k\|}.$$

Now, $k\in K'$ is an unsuccessful iteration and $x_k\in L$. By
Proposition~\ref{prop_dominance},
\begin{equation}\label{ineq}\exists j(k)\in\{1,\ldots,m+1\}:f_{j(k)}(x_k+\alpha_k d_k)>f(x_k)-\bar{\rho}(\alpha_k).\end{equation}
Since the number of components of the objective function is
finite, by passing to a subsequence $K''\subseteq K'$ that always
uses the same component of $\bar{F}$ we have the desired result.
\end{proof}

The convergence to Pareto-Clarke or Pareto-Clarke-KKT critical
points can be proven by imposing asymptotic density in the unit
sphere of the set of refining directions associated with
$\bar{x}$.

\begin{theorem}
\label{opt_f} Consider $\{x_k\}_{k\in K}$, a refining subsequence
generated by DMS-FILTER-IR, converging to the refined point
$\bar{x}\in \bar{X}$. Assume that
$H_{\bar{X}}^{Cl}(\bar{x})\neq\emptyset$ and that $\bar{F}$ is
Lipschitz continuous near $\bar{x}$. If the set of refining
directions for $\bar{x}$ is dense in $T^{Cl}_{\bar{X}}(\bar{x})$,
then $\bar{x}$ is a Pareto-Clarke critical point for
Problem~(\ref{MOO2}). If, in addition, $F$ is strictly
differentiable at $\bar{x}$, then this point is a
Pareto-Clarke-KKT critical point for Problem~(\ref{MOO2}).
\end{theorem}
\begin{proof}
The proof follows from Theorem~\ref{teo_opt_f}, using similar
arguments to the ones of Theorem 4.9 in~\cite{custodio2011}.
\end{proof}

The previous theorem states that DMS-FILTER-IR generates a
Pareto-Clarke critical point for Problem~(\ref{MOO2}). However, if
$\bar{x}$ is a local efficient point of Problem~(\ref{MOO2}) and
$h(\bar{x})=0$, then $\bar{x}$ is a local efficient point for
Problem~(\ref{MOO}).

Thus, a question that arises is if the algorithm is indeed able to
generate feasible points, when initialized from infeasible ones.
The next theorem attempts to establish some conditions that
provide an answer to this question.

\begin{theorem}
Let Assumption~\ref{assumption1} hold. Assume that $h$ is a
continuous function and consider $\{x_k\}_{k\in K}$, an infeasible
refining subsequence, such that for each $k\in K$, $x_k$ is used
at a successful inexact restoration step in DMS-FILTER-IR. Then,
the algorithm generates a limit point $\bar{y}\in \Upsilon$.
\end{theorem}
\begin{proof}
For each $k\in K$, $x_k$ was used at a successful inexact
restoration step. Thus, there is $y_k^*\in X$ such that $y_k^*$ is
nondominated and $0\leq h(y_k^*)\leq \xi(\alpha_k)h(x_k)\leq
h_{\tt max}$.

Since $\{x_k\}_{k\in K}$ is a refining subsequence, $\lim_{k\in
K}\alpha_k=0$. Considering the properties of $\xi(\cdot)$ and the
boundedness of $h$, we can conclude $\lim_{k\in K}h(y_k^*)=0$.
Assumption~\ref{assumption1} allows to consider $K'\subseteq K$
such that $\lim_{k\in K}y_k^*=\bar{y}$ and the continuity of $h$
establishes $h(\bar{y})=0$, meaning that $\bar{y}\in \Upsilon$.
\end{proof}

The following result links sequences of points generated by
DMS-FILTER-IR to Problem~(\ref{MOO}). For establishing it, we will
assume that globalization is based on the use of integer lattices.
\begin{corollary}\label{corol:feasible}
Consider $\{x_k\}_{k\in K}$ a feasible refining subsequence,
converging to the refined point $\bar{x}\in \Upsilon$, generated
by algorithm DMS-FILTER-IR, when using a globalization strategy
based on integer lattices. Assume that $\bar{F}$ is Lipschitz
continuous near $\bar{x}$. Let $d\in H^{Cl}_{\Upsilon}(\bar{x})$
be a refining direction for $\bar{x}$, associated with the
refining subsequence $\{x_k\}_{k\in K}$. Then,
$$\exists j\in\set{1,\ldots,m}: f_j^\circ(\bar{x};d)\geq
0.$$
\end{corollary}
\begin{proof}
The proof follows directly from the proof of
Theorem~\ref{teo_opt_f}, considering $d\in
H^{Cl}_{\Upsilon}(\bar{x})$ and noting that for $k$ sufficiently
large $h(x_k+\alpha_k d_k)=0=h(x_k)$. Thus,
inequality~(\ref{ineq}) needs to hold for $j\in\set{1,\ldots,m}$.
\end{proof}

In this situation, $\bar{x}$ is a Pareto-Clark critical point for
Problem~(\ref{MOO}).

\begin{corollary}
Consider $\{x_k\}_{k\in K}$, a feasible refining subsequence
generated by DMS-FILTER-IR, converging to the refined point
$\bar{x}\in \Upsilon$. Assume that
$H_{\Upsilon}^{Cl}(\bar{x})\neq\emptyset$ and that $\bar{F}$ is
Lipschitz continuous near $\bar{x}$. If the set of refining
directions for $\bar{x}$ is dense in $T^{Cl}_{\Upsilon}(\bar{x})$,
then $\bar{x}$ is a Pareto-Clarke critical point for
Problem~(\ref{MOO}). If, in addition, $F$ is strictly
differentiable at $\bar{x}$, then this point is a
Pareto-Clarke-KKT critical point for Problem~(\ref{MOO}).
\end{corollary}
\begin{proof}
The proof is a direct consequence of
Corollary~\ref{corol:feasible}, using similar arguments to the
ones of Theorem 4.9 in~\cite{custodio2011}.
\end{proof}

By using a filter-based approach, we have overcome the difficulty
of applying directional direct search to multiobjective
constrained optimization problems, when a feasible initialization
is not available (regarding the relaxable constraints). The
incorporation of a inexact restoration step potentiates
feasibility. The next section will illustrate the numerical
competitiveness of the proposed approach.


\section{Implementation Details}
\label{sec:implementation}

Differently from DMulti-MADS-PB~\cite{bigeon2022}, DMS-FILTER-IR
selects a single point at each iteration, to be explored at the
search and, possibly, at the inexact restoration and/or pool
steps. The decision on using a feasible or infeasible iterate
point always attempts to promote feasibility, regarding the
relaxable constraints (unrelaxable constraints are always
satisfied by the points in the list, from where the iterate point
will be selected).

The algorithm switches to an infeasible iterate point if the
current feasible iterate point only generates infeasible points.
At the next iteration, the iterate point will be selected from the
nondominated points in the list that do not satisfy the relaxable
constraints. Once that an infeasible iterate point generates at
least one feasible point, a feasible point will be selected as
iterate point for the next iteration.

Suppose that we are at one iteration where the iterate point
should be feasible. For selecting it from all the feasible points
in the iterate list, we use the concept of the most isolated
point. For each component of the objective function
$i=1,\ldots,m$, the feasible points in $L_k$ are selected and
ordered by increasing function value:
$$f_i(x^1)\leq f_i(x^2)\leq\ldots\leq f_i(x^{m_F}),$$ where $m_F$
denotes the total number of feasible points in $L_k$.

Then, for each component of the objective function $f_i$,
$i=1,\ldots,m$, and for each feasible point $x^j$ in $L_k$, with
$j=1,\ldots,m_F$, the following indicator is computed:
\begin{equation*}
\delta_{i}\left(x^{j}\right)= \begin{cases}
f_{i}\left(x^{2}\right)-f_{i}\left(x^{1}\right) & \text {, if }
j=1
\\ f_{i}\left(x^{m_F}\right)-f_{i}\left(x^{m_F-1}\right) & \text {, if } j=m_F \\
\dfrac{f_{i}\left(x^{j+1}\right)-f_{i}\left(x^{j-1}\right)}{2} &
\text {, otherwise. }\end{cases}
\end{equation*}
The most isolated point corresponds to the maximum value of
$\gamma_j$, that is, $$x_k \in\displaystyle
\argmax_{j=1,\ldots,m_F} \gamma_j,$$ where
\begin{equation}
\gamma_{j} = \displaystyle \dfrac{1}{m}\sum_{i=1}^{m}
\delta_i(x^j).
\end{equation}

When the iterate point should be infeasible, two different
criteria are used for its selection, depending on having at least
one feasible point in the list or not. In the latter situation,
the point selected corresponds to the most promising infeasible
point is the list, in terms of restoring the feasibility. Thus,
the point with the smallest value for the aggregated constraint
violation function $h$ will be chosen.

Now, suppose that there is at least one feasible point in the
list. The fact that the iterate point should be infeasible means
that at the last iteration a feasible iterate point only generated
infeasible ones. Thus, we want to try to restore feasibility close
to the region that was being explored. A closed ball centered on
the feasible iterate point, of radius equal to $\eta\,\alpha_k
\max_{d\in D_{k-1}}\|d\|$, with $\eta\geq \frac{1}{\beta_2}$, is
considered and the infeasible point in the list, belonging to this
ball, with the lowest value for the aggregated constraint
violation function $h$ will be selected as iterate point.

DMS-FILTER-IR was implemented in \texttt{Matlab}, keeping the
default settings of DMS. Thus, the step size was initialized as
$1$, kept constant at successful iterations ($\gamma = 1$), and
halved at unsuccessful ones ($\beta_1=\beta_2 = 0.5$). No search
step was implemented.

The aggregated violation function was defined as:
\begin{equation*}
h(x)=\|C(x)_+\|_2^2=\displaystyle\sum_{i=1}^{p}\max\set{0,c_i(x)}^2.
\end{equation*} Regarding the maximum violation allowed, $h_{\tt max}$, it
depends on the initialization. If there are any infeasible points
in the list, $h_{\tt max}$ will be set equal to the largest of the
existing values of $h$. Otherwise, it will be set equal to the
maximum between $10$ and half of the number of the relaxable
constraints.

Many options can be considered for defining a function complying
with the requirements of $\xi(\cdot)$, to be used in the inexact
restoration step. In the numerical implementation, we used the
function $\xi(\alpha)=(\frac{\alpha}{2})^2$. Note that this
function is continuous, $\xi(\alpha)\to 0$ when $\alpha\downarrow
0$ and, considering the initialization and the strategy to update
the step size, $0<\xi(\alpha)<1$. The inexact restoration problems
were solved with the \texttt{Matlab} function
$\texttt{fmincon.m}$. Through the numerical section, for any
solver, feasibility is assumed to be achieved when there is an
aggregated violation of the relaxable constraints less than
$10^{-5}$.

If the poll step is performed, a complete polling approach is
adopted, evaluating all the points corresponding to directions in
the positive spanning set.


\section{Numerical Experiments}
\label{sec:numerical_experiments} This section is devoted to the
numerical experiments performed with DMS-FILTER-IR, illustrating
its numerical behavior by comparison with the original DMS
algorithm, when solving constrained problems, or with other
state-of-art solvers. A first subsection will describe the metrics
used in the performance assessment, followed by a new subsection
detailing the problem collection. The last two subsections
illustrate the numerical behavior of the proposed algorithm. All
tests were performed in a laptop with a 11th Gen
Intel{\small\textregistered}\ Core(TM) i7-1165G7 processor, at
2.80GHz, with 16GB of RAM memory, using Windows 11 with 64 bits.

\subsection{Performance Assessment}

In order to evaluate the numerical performance of the algorithm,
we resource to performance profiles, a tool introduced by Dolan
and Mor\'e~\cite{Dolan_More_2002}, for single objective nonlinear
optimization. This tool allows to concurrently assess the
numerical performance of multiple solvers, for different metrics.
The performance of a solver $s$, belonging to a set of solvers
$S$, on a particular set of problems $P$ is represented by the
cumulative function:
\begin{equation*}
\rho_s(\tau)=\dfrac{1}{|P|}\left|\left\{p\in P\mid t_{p,s}\leq
\tau \min\set{t_{p,s}: s\in S}\right\}\right|,
\end{equation*}
where $\tau\geq 1$ and $t_{p,s}$ represents the value of the
selected metric, obtained by solver $s\in S$ on problem $p\in P$.
It is assumed that lower values of $ t_{p,s}$ correspond to a
better values for the metric.

Higher values of $\rho_s(\tau)$ represent a better numerical
performance for solver $s$. Specifically, the solver with the
highest $\rho_s(1)$ value is considered the most efficient and the
solver with the greatest $\rho_s(\tau)$ value for large values of
$\tau$ is regarded as the most robust.

As metrics, purity, hypervolume, and the spread metrics $\Gamma$
and $\Delta$ were selected. The percentage of nondominated points
generated by a given solver is measured by purity:
$$\bar{t}_{p,s}={Pur}_{p,s}=\frac{|F_{p,s} \cap F_{p}|}{|F_{p,s}|},$$
where $ F_{p,s} $ represents the approximation to the Pareto front
of problem $p$ computed by solver $s$ and $ F_{p} $ is a reference
Pareto front for problem $p$, computed by considering the union of
the Pareto approximations corresponding to all solvers, $ \cup_{s
\in S} F_{p,s} $, and removing from it all the dominated
points~\cite{custodio2011}.

Hypervolume~\cite{EZitzler_et_al_2003}, additionally to
nondominance, encompasses the notion of spread by quantifying the
volume of the region dominated by the current approximation of the
Pareto front and limited by a reference point $U_p \in
\mathbb{R}^m$, dominated by all points belonging to the different
approximations computed for the Pareto front of problem $p\in P$
by all solvers tested. In a formal way:

$$\bar{t}_{p,s}=HV_{p,s} =  Vol\{y \in \mathbb{R}^m\,| \, y \le U_p \wedge \exists x \in F_{p,s} : x \le y\} = Vol \left(\bigcup_{x \in F_{p,s}} [x, U_p]\right),$$
where $Vol(.)$ denotes the Lebesgue measure of a $m$-dimensional
set of points and $[x, U_p]$ denotes the interval box with lower
corner $x$ and upper corner $U_p$.

Since larger values of purity and hypervolume indicate a better
performance, the inverse value of each one of the metrics was used
($t_{p,s}=1/ \bar{t}_{p,s}$), when computing the associated
performance profiles.

Lastly, for a direct evaluation of spread along the Pareto front,
we incorporated two supplementary metrics: the $\Gamma$ metric,
quantifying the magnitude of the largest gap in the computed
Pareto front approximation, and the $\Delta$ metric, which gauges
the evenness of the distribution of nondominated points within the
generated approximation.

Consider that solver $s \in S$ computed for problem $p \in P$ an
approximation to the Pareto front with points
$x^1,x^2,\ldots,x^N$, to which we add the so-called extreme
points, $x^0$ and $x^{N+1}$, corresponding to the points with the
best and worst values for each component of the objective
function. The metric $\Gamma$ can be computed as:
\begin{equation}
 \Gamma_{p,s} \; = \; \max_{j \in
\{1,\dots,q\}}\left(\max_{i\in
\{0,\dots,N\}}\{\delta_{j,i}\}\right),
\end{equation}
\noindent where $\delta_{j,i}=f_{j}(x^{i+1})-f_{j}(x^i)$, assuming
that the objective function values have been sorted by increasing
order for each objective function component~$j$. The metric
$\Delta$~\cite{KDeb_et_al_2002} is computed as:
\begin{equation}\displaystyle \Delta_{p,s} \;
= \;
\max_{j\in\{1,\dots,q\}}\left(\frac{\delta_{j,0}+\delta_{j,N}+\sum_{i=1}^{N-1}|\delta_{j,i}-
\bar{\delta}_j|}{\delta_{j,0}+\delta_{j,N}+(N-1)\bar{\delta}_j}\right),
\end{equation}
where $\bar{\delta}_j$, for $j=1,\ldots,q$, represents the average
of the distances $\delta_{j,i}$, $i=1,\dots,N-1$.

\subsection{Problem Collection}
Liuzzi et al.~\cite{liuzzi2016} defined a collection of
constrained problems by coupling a subset of the bound constrained
problems provided in Cust\'odio et al.~\cite{custodio2011} with
six families of constraints proposed in~\cite{karmitsa}. All the
bound problems in~\cite{custodio2011} with $n\geq 3$ variables
were selected, resulting in $51$ bound constrained problems. A set
of $306$ constrained problems was generated by adding to each
problem the following six families of nonlinear constraints (the
suggested initialization is denoted by $x^0$):
\begin{eqnarray*}
g_j^1(x) &=& (3-2x_{j+1})x_{j+1}-x_j-2x_{j+2}+1,\quad j=1,\ldots,p,\quad p=n-2\\
& & x^0_i=1,\quad i=1,\ldots,n\nonumber\\
g_j^2(x) &=& (3-2x_{j+1})x_{j+1}-x_j-2x_{j+2}+2.5,\quad j=1,\ldots,p,\quad p=n-2\\
& & x^0_i=2,\quad i=1,\ldots,n\nonumber\\
g_j^3(x) &=& x_j^2+x_{j+1}^2+x_j x_{j+1}-2x_j-2x_{j+1}+1,\quad j=1,\ldots,p,\quad p=n-1\\
& & x^0_i=0.5\quad i=1,\ldots,n\nonumber\\
g_j^4(x) &=& x_j^2+x_{j+1}^2+x_j x_{j+1}-1,\quad j=1,\ldots,p,\quad p=n-1\\
& & x^0_i=0,\quad i=1,\ldots,n\nonumber\\
g_j^5(x) &=& (3-0.5x_{j+1})x_{j+1}-x_j-2x_{j+2}+1,\quad j=1,\ldots,p,\quad p=n-2,\\
& & x^0_i=2,\quad i=1,\ldots,n\nonumber\\
g_j^6(x) &=& \displaystyle \sum_{i=1}^{n-2} ((3-0.50x_{i+1})x_{i+1}-x_i-2x_{i+2}+1),\quad j=p,\quad p=1\\
& & x^0_i=2,\quad i=1,\ldots,n\nonumber
\end{eqnarray*}
In theory, DMS is developed for multiobjective derivative-free
optimization, regardless of the number of components in the
objective function. However, the numerical results reported
in~\cite{custodio2011} respect to biojective and triobjective
problems (there is an exception of one problem with four
components in the objective function, from the set of 100 bound
constrained problems considered). In fact, it is common knowledge
in the multiobjective optimization community that addressing
problems with more than three components in the objective function
requires special techniques, falling on the specific domain of
Many-objective Optimization~\cite{emmerich2018}. The focus of this
work is not Many-objective Optimization problems. Considering that
the aggregated penalization function will be an additional
component of the objective function, we restricted the set of
$306$ constrained problems to biobjective ones, in a total of
$156$. The problems and their dimensions are given in
Table~\ref{biobjective_problems}.

\begin{table}[h!]
\centering
\begin{tabular}{@{}llllll@{}}
\toprule \multicolumn{1}{c}{\textbf{Problem}} &
\multicolumn{1}{c}{\textbf{n}} &
\multicolumn{1}{c}{\textbf{Problem}} &
\multicolumn{1}{c}{\textbf{n}} &
\multicolumn{1}{c}{\textbf{Problem}} &
\multicolumn{1}{c}{\textbf{n}} \\ \midrule
CL1 & 4 & L2ZDT6 & 10 & QV1 & 10 \\
DPAM1& 10 & L3ZDT1 & 30 & ZDT1 & 30\\
FES1 & 10 & L3ZDT2 & 30 & ZDT2 & 30\\
Kursawe& 3 & L3ZDT3 & 30 & ZDT3& 30\\
L1ZDT4& 10& L3ZDT4& 30 & ZDT4  & 10\\
L2ZDT1& 30& L3ZDT6 & 10 & ZDT6& 10\\
L2ZDT2 & 30 & MOP2 & 4 & SK2 & 4 \\
L2ZDT3 & 30 & MOP4 & 3 & TKLY1 & 4 \\
L2ZDT4 & 30 & OKA2 & 3 & &\\
\bottomrule
\end{tabular}
\caption{The biobjective test set used in the numerical
experiments ($n$ denotes the number of variables).}
\label{biobjective_problems}
\end{table}

DMS requires a feasible initialization. However, not all the
points $x^0$ provided by Karmitsa~\cite{karmitsa} satisfy the
bounds constraints. Thus, from the $156$ biobjective problems
considered, we retained only the $93$ for which a feasible
initialization was available. The final test set comprise $93$
nonlinearly constrained biobjective problems, with $n\in [3,30]$
and $p\in [1,29]$. We assumed the nonlinear constraints as being
relaxable, corresponding the unrelaxable constraints to bounds.

\subsection{Positive Spanning Sets}

The convergence of both DMS and DMS-FILTER-IR is established under
the assumption of asymptotic density of the sets of directions
used by the algorithm, during the optimization process. However,
when the budget of function evaluations is limited, which is often
the case when the function is expensive to evaluate, this density
is never accomplished. Coordinate search has the perfect geometry
for bound constrained problems. In fact, this was the type of
directions used to obtain the numerical results reported
in~\cite{custodio2011}, illustrating the numerical performance of
DMS.

Considering that in DMS-FILTER-IR the nonlinear constraints are
addressed by the filter approach, using the aggregated violation
function and reducing the problem to a bound constrained problem,
it would be interesting to compare DMS and DMS-FILTER-IR when
using coordinate search as set of directions or when resorting to
an asymptotic dense set of directions, built using the technique
proposed in~\cite{MAAbramson_etal_2009}, based on Halton
sequences.

DMS addresses constraints with an extreme barrier approach,
denoted here by DMS-EB, only evaluating feasible points and
requiring a feasible initialization. Thus, the feasible point
$x^0$ given in~\cite{karmitsa} was used for initialization.
DMS-FILTER-IR allows infeasible points, with respect to the
relaxable constraints. Therefore, DMS-FILTER-IR was initialized
with $n$-points equally spaced in the line segment, joining the
variable upper and lower bounds.

Figures~\ref{performance_DMS 5000funceval}
and~\ref{performance_DMS-FILTER-IR 5000funceval} report the
results obtained for DMS and DMS-FILTER-IR, respectively. A
maximum budget of $5000$ function evaluations was allowed, jointly
with a minimum step size of $10^{-3}$, for all the points in the
list.

\begin{figure}[h!]
\centering \subfigure[Purity
\label{purity_500}]{\includegraphics[scale=0.4]{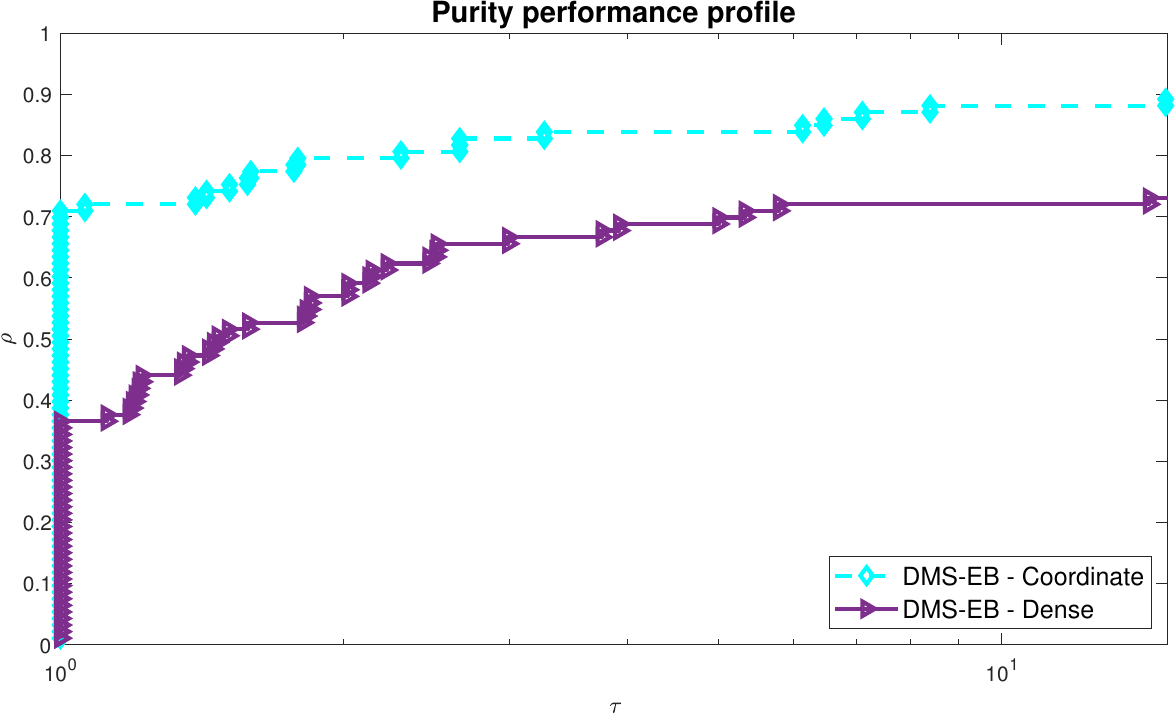}}
\subfigure[Hypervolume
\label{hypervolume_500}]{\includegraphics[scale=0.4]{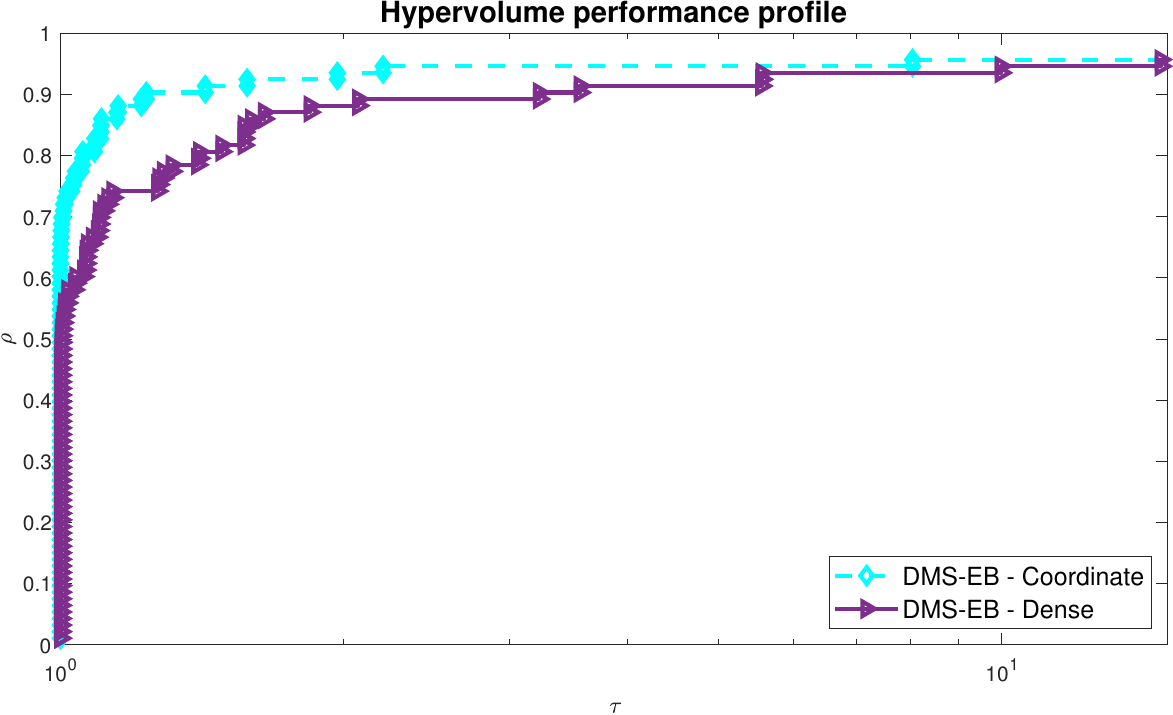}}
\subfigure[Spread Gamma
\label{gamma_500}]{\includegraphics[scale=0.4]{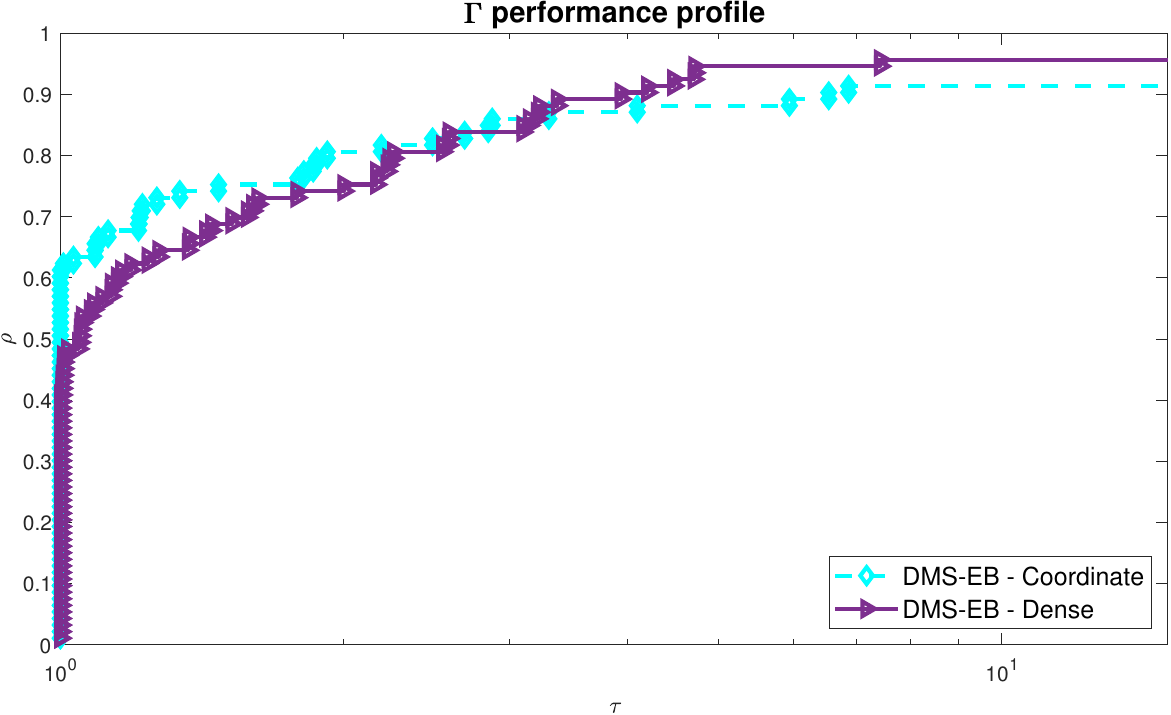}}
\subfigure[Spread Delta
\label{delta_500}]{\includegraphics[scale=0.4]{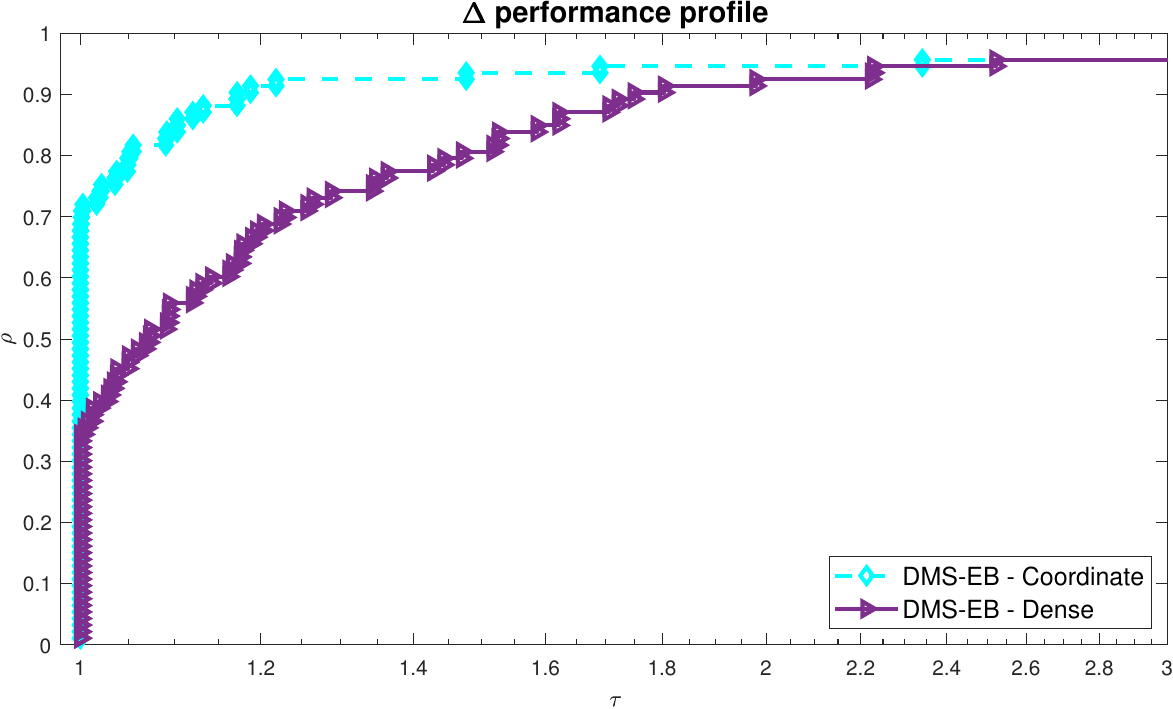}}
\caption{Performance profiles for  DMS considering different types
of positive spanning sets and a maximum budget of $5000$ function
evaluations.} \label{performance_DMS 5000funceval}
\end{figure}

\begin{figure}[h!]
\centering \subfigure[Purity
\label{purity_500}]{\includegraphics[scale=0.4]{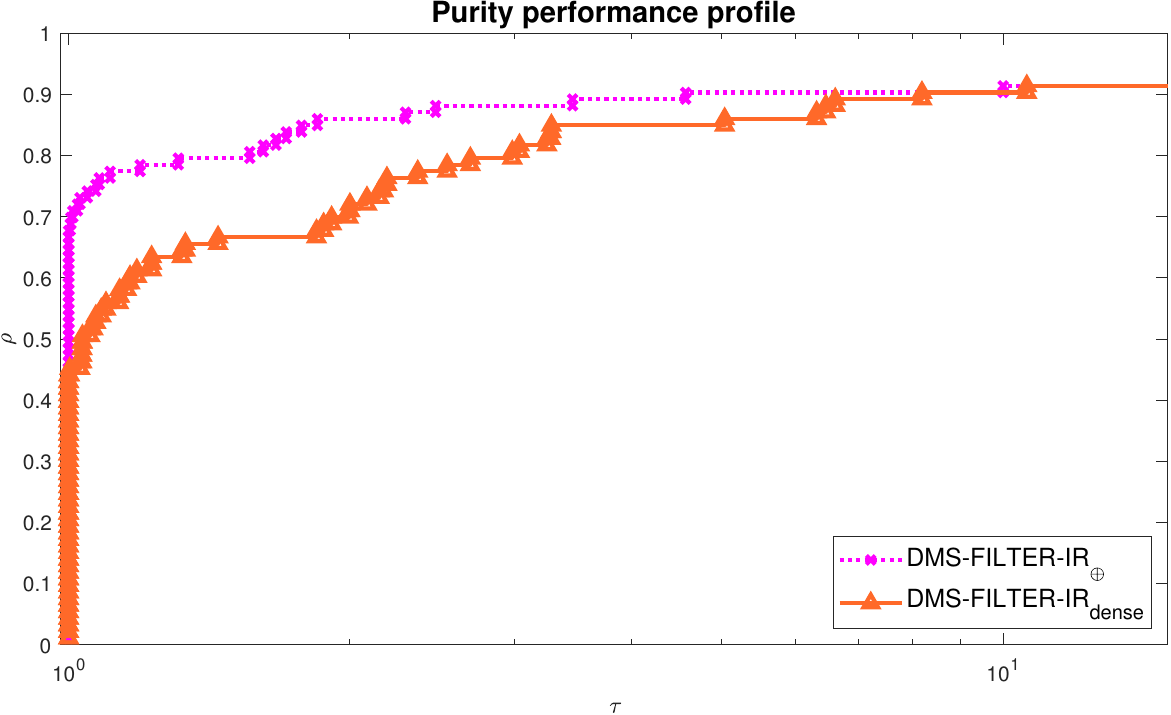}}
\subfigure[Hypervolume
\label{hypervolume_500}]{\includegraphics[scale=0.4]{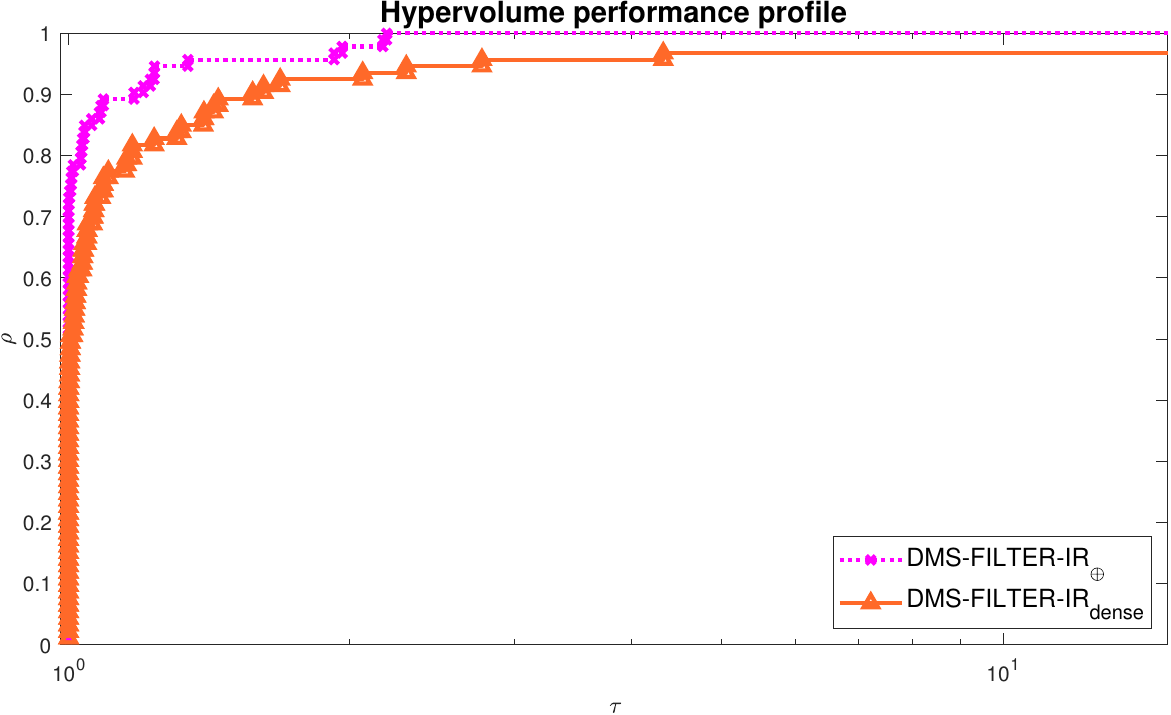}}
\subfigure[Spread Gamma
\label{gamma_500}]{\includegraphics[scale=0.4]{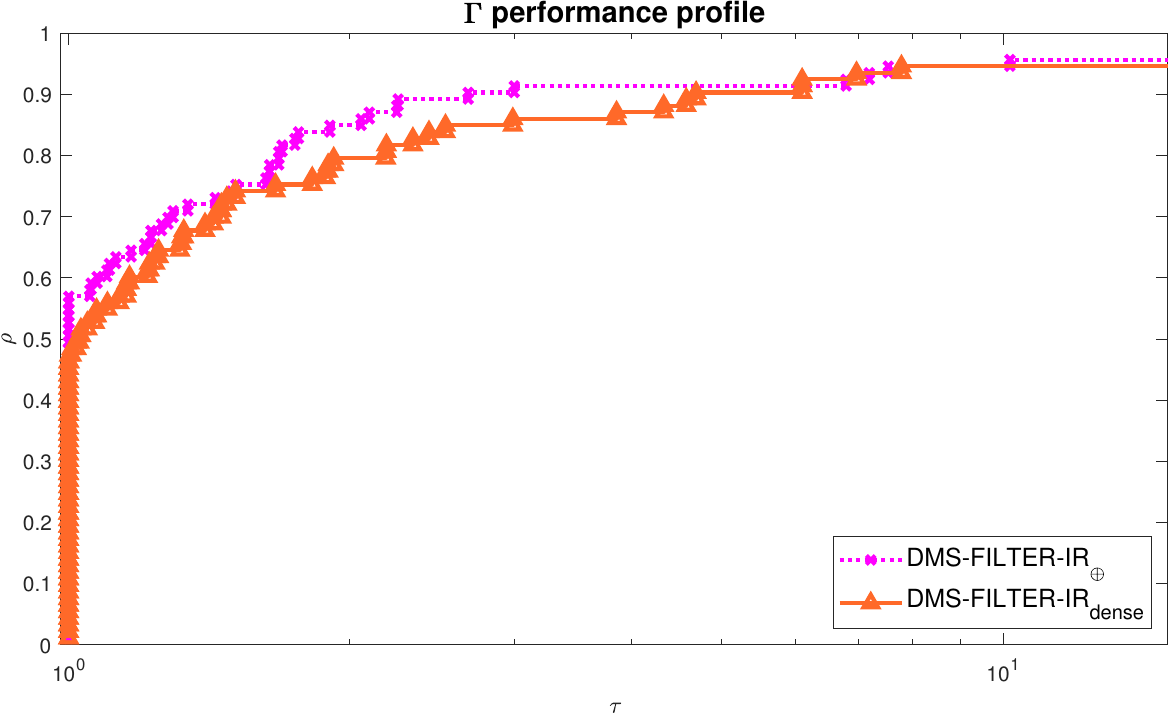}}
\subfigure[Spread Delta
\label{delta_500}]{\includegraphics[scale=0.4]{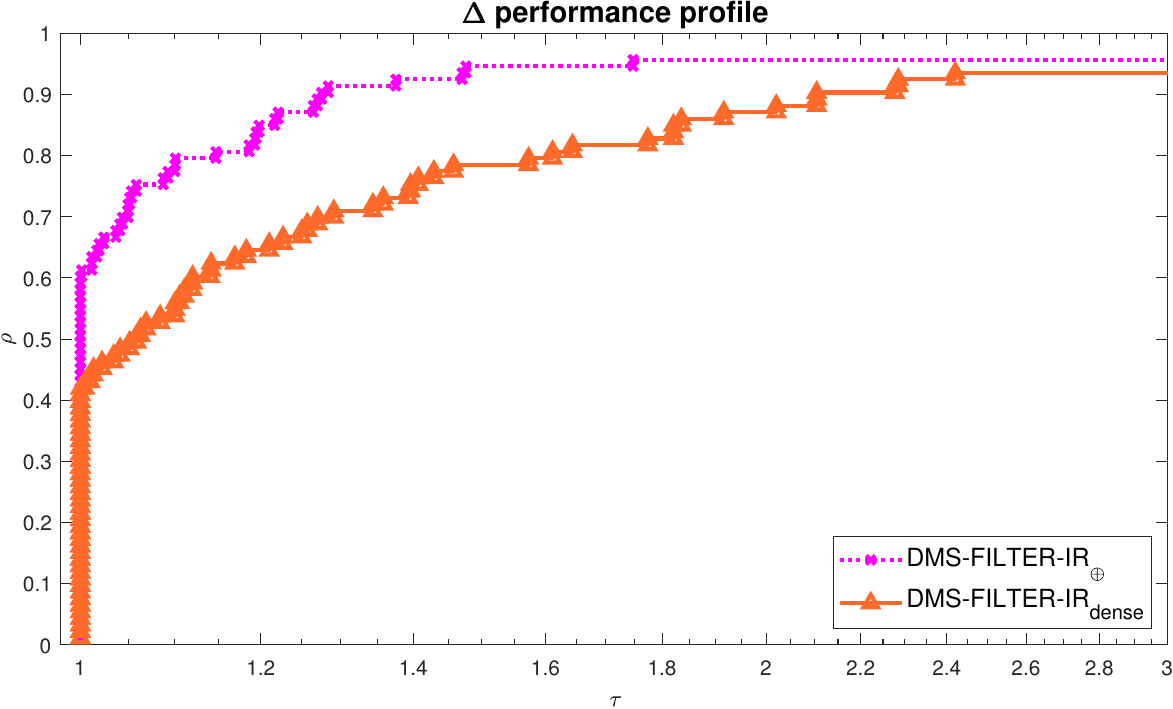}}
\caption{Performance profiles for DMS-FILTER-IR considering
different types of positive spanning sets and a maximum budget of
$5000$ function evaluations.} \label{performance_DMS-FILTER-IR
5000funceval}
\end{figure}

With exception to the spread metric $\Gamma$, where the results
are very close, for any of the metrics considered, it is clear the
advantage of the use of coordinate directions as positive spanning
set, both for DMS and DMS-FILTER-IR. As already mentioned, this
could be the result of the nice geometry associated to these
directions and bound constrained problems. Thus, in the next
section, DMS and DMS-FILTER-IR will consider positive spanning
sets based on coordinate search.

\subsection{Comparison With Other Solvers}
In addition to DMS~\cite{custodio2011}, DFMO~\cite{liuzzi2016},
and DMultiMADS-PB~\cite{bigeon2022} were also tested as benchmark
solvers to evaluate the performance of DMS-FILTER-IR. The DMS
solver is implemented in {\tt Matlab} and is freely available at
\url{http://www.mat.uc.pt/dms}. DFMO is coded in {\tt Fortran90}
and is available at \url{http://www.dis.uniroma1.it/~lucidi/DFL}.
Finally, coded in {\tt Julia}, DMultiMADS-PB can be obtained from
\url{https://github.com/bbopt/DMultiMadsPB}.

The DFMO and DMultiMADS-PB algorithms address nonlinearly
constrained multiobjective optimization problems by penalizing the
nonlinear constraints with an exact merit function or by using a
progressive barrier approach, respectively. As already mentioned,
DMS addresses constraints with an extreme barrier function.

All solvers were run with the defaults, selecting the best version
identified in~\cite{liuzzi2016} for DFMO and using the best
version reported in~\cite{bigeon2022} for DMultiMADS-PB. Results
were obtained for maximum budgets of $500$ and $5000$ function
evaluations. Considering the expensive nature of the function
evaluation, small budgets are particular relevant for evaluating
the performance of the solvers.

As in the previous subsection, DMS-EB was initialized with the
feasible point $x^0$ provided in~\cite{karmitsa}. DMS-FILTER-IR,
DMultiMADS-PB, and DFMO can be initialized with infeasible points.
Therefore, DMS-FILTER-IR and DMultiMADS-PB were initialized with
$n$-points equally spaced in the line segment, joining the
variable upper and lower bounds, which is the default
initialization of DMultiMADS-PB. DFMO was initialized with the
centroid of the box defined by the bound constraints. After, the
algorithm also generates $n$-points equally spaced in the line
segment joining the bounds.

Figure~\ref{performance_500funceval} depicts performance profiles
for the different metrics considered, when a maximum budget of
$500$ function evaluations is allowed.

\begin{figure}[h!]
\centering \subfigure[Purity
\label{purity_500}]{\includegraphics[scale=0.4]{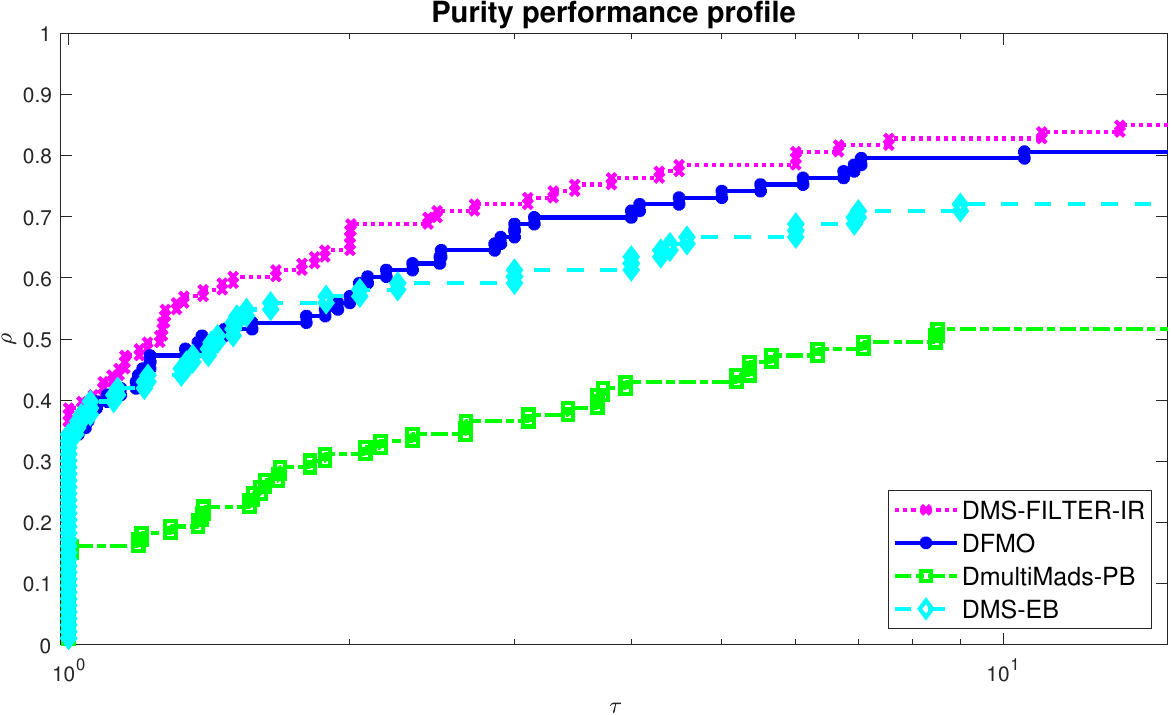}}
\subfigure[Hypervolume
\label{hypervolume_500}]{\includegraphics[scale=0.4]{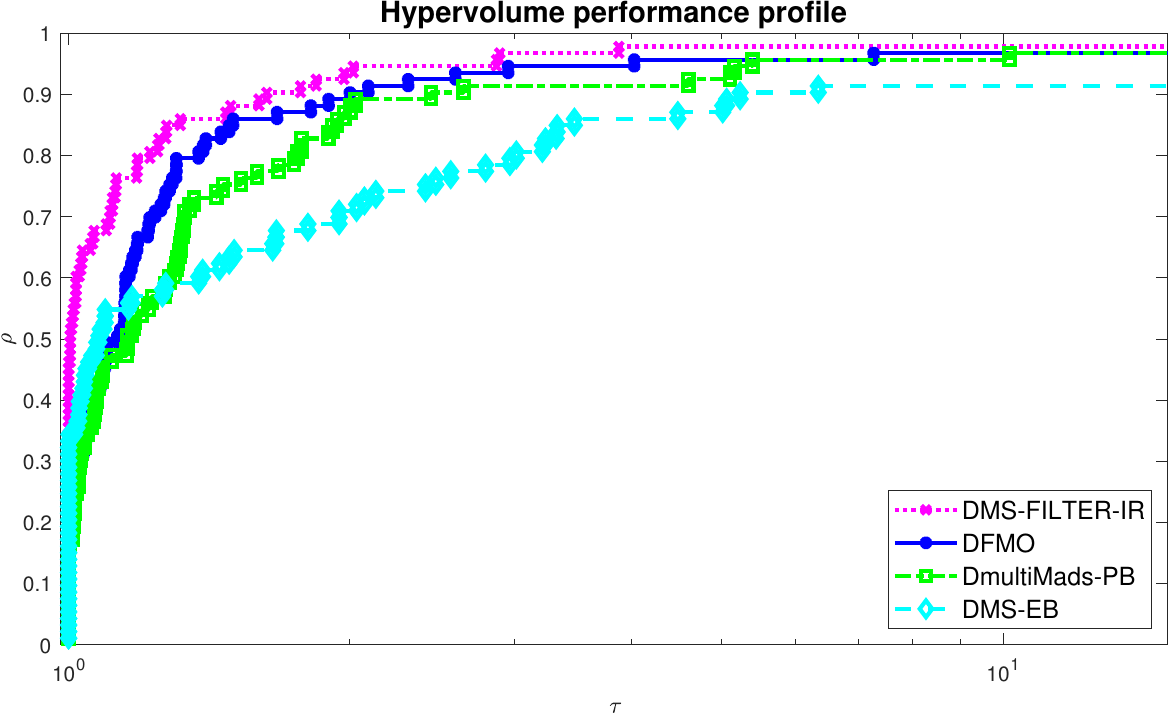}}
\subfigure[Spread Gamma
\label{gamma_500}]{\includegraphics[scale=0.4]{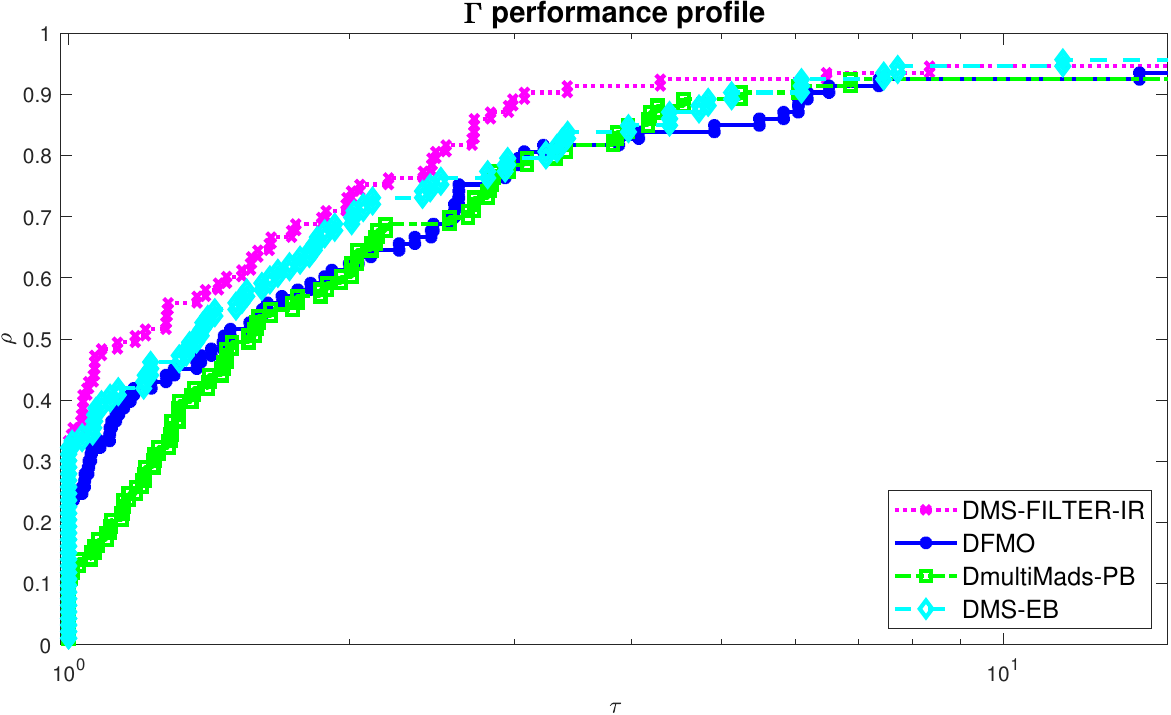}}
\subfigure[Spread Delta
\label{delta_500}]{\includegraphics[scale=0.4]{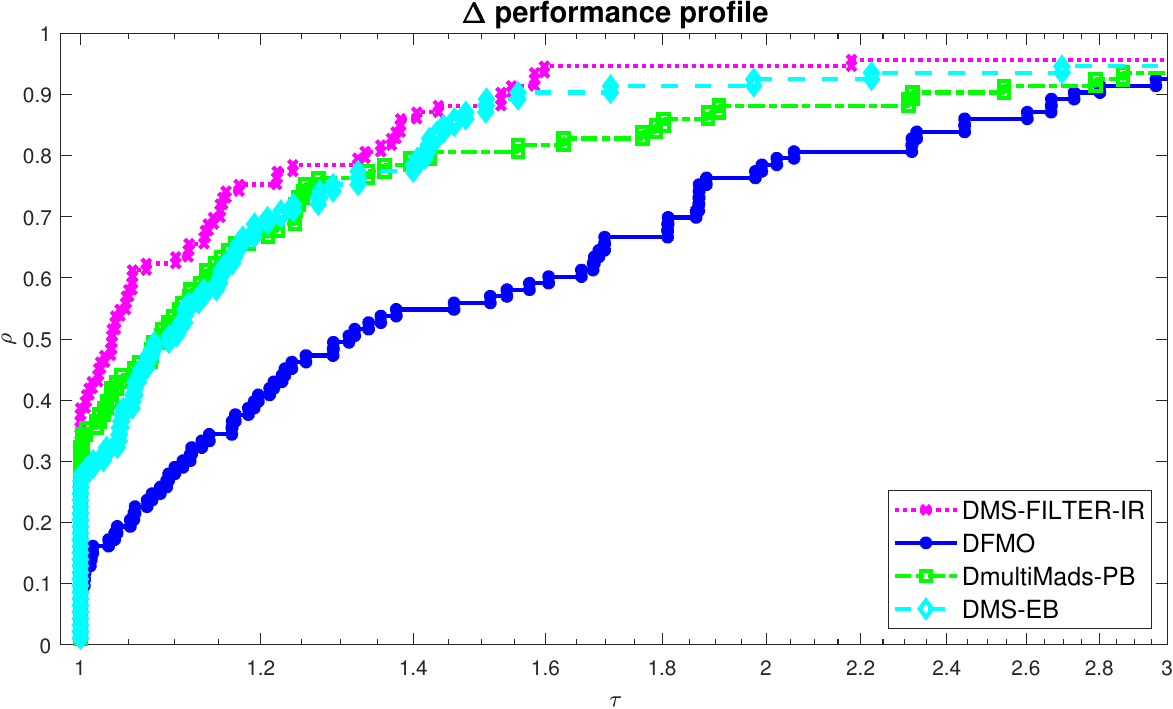}}
\caption{Comparing DMS-FILTER-IR with DFMO, DMultiMADS-PB, and DMS
based on performance profiles for a maximum of $500$ function
evaluations.} \label{performance_500funceval}
\end{figure}

In general, DMS-FILTER-IR presents a good performance for any of
the four metrics considered. Noteworthy, it is the most efficient
solver for hypervolume and presents some advantage regarding
robustness for the purity metric.

When the maximum budget allowed increases to $5000$ function
evaluations (see Figure~\ref{performance_5000funceval}),
DMS-FILTER-IR remains as the most competitive solver in what
respects purity and hypervolume. For the spread metric $\Delta$,
DMS-EB presents a better performance. DFMO provides some good
results in terms of the largest gap in the Pareto front,
represented by metric $\Gamma$.

\begin{figure}[h!]
\centering
\subfigure[Purity]{\includegraphics[scale=0.4]{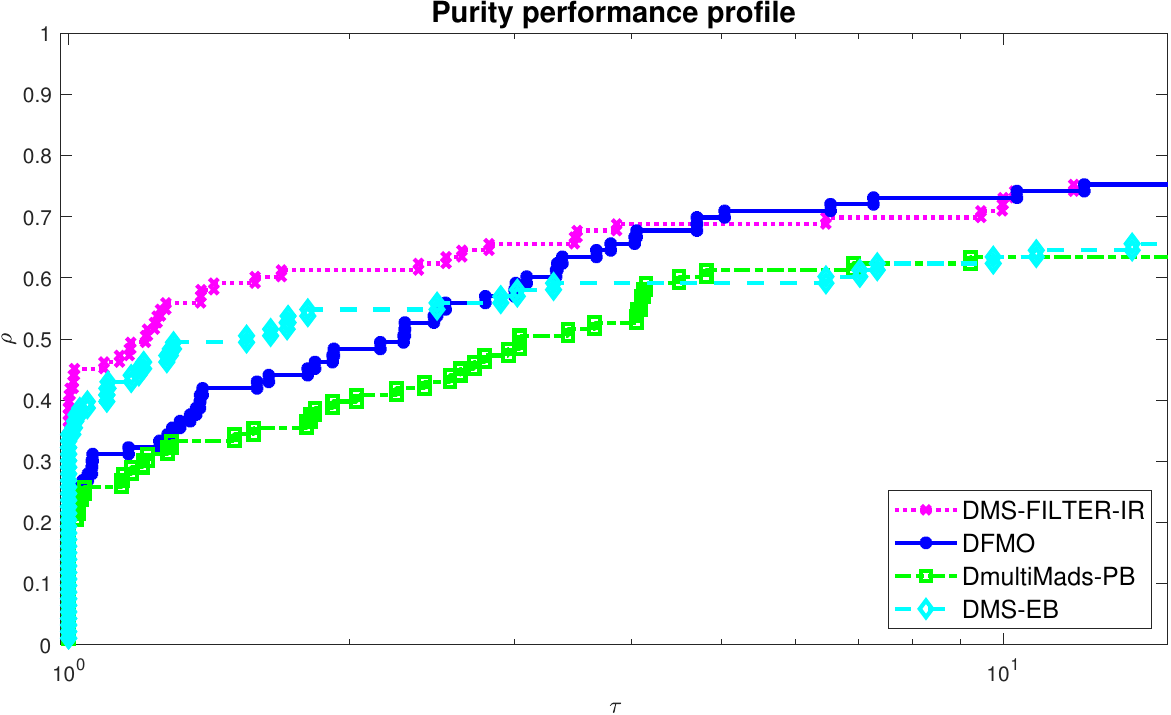}}
\subfigure[Hypervolume]{\includegraphics[scale=0.4]{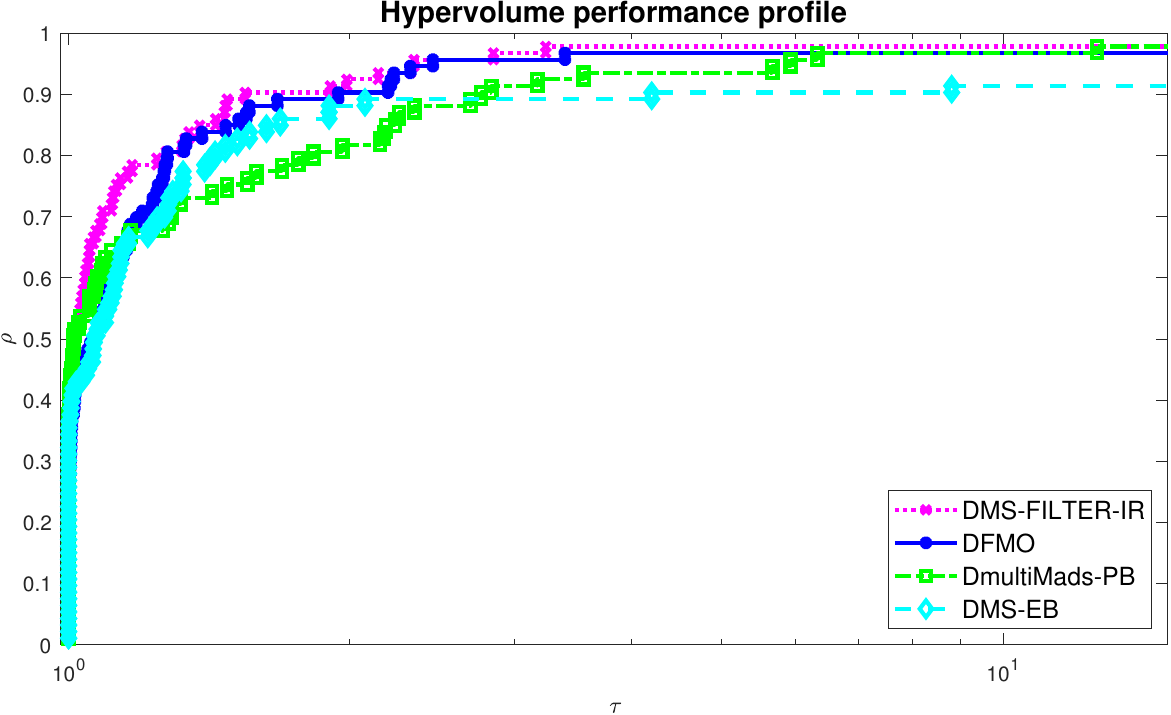}}
\subfigure[Spread
Gamma]{\includegraphics[scale=0.4]{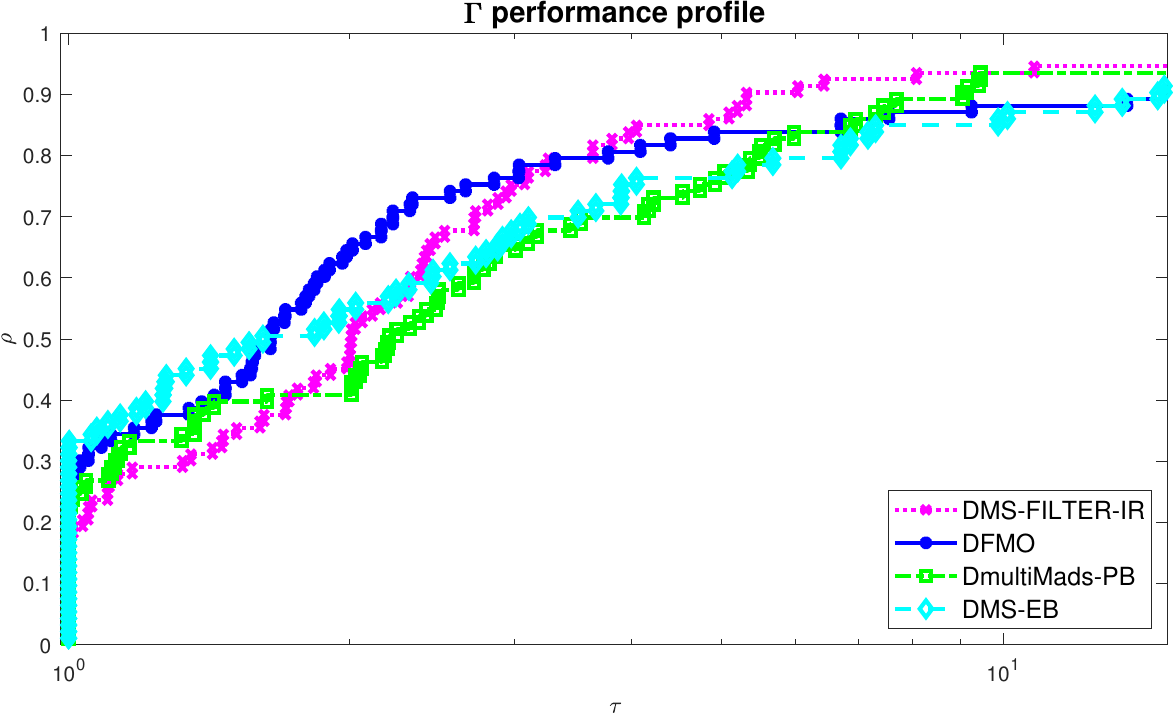}}
\subfigure[Spread
Delta]{\includegraphics[scale=0.4]{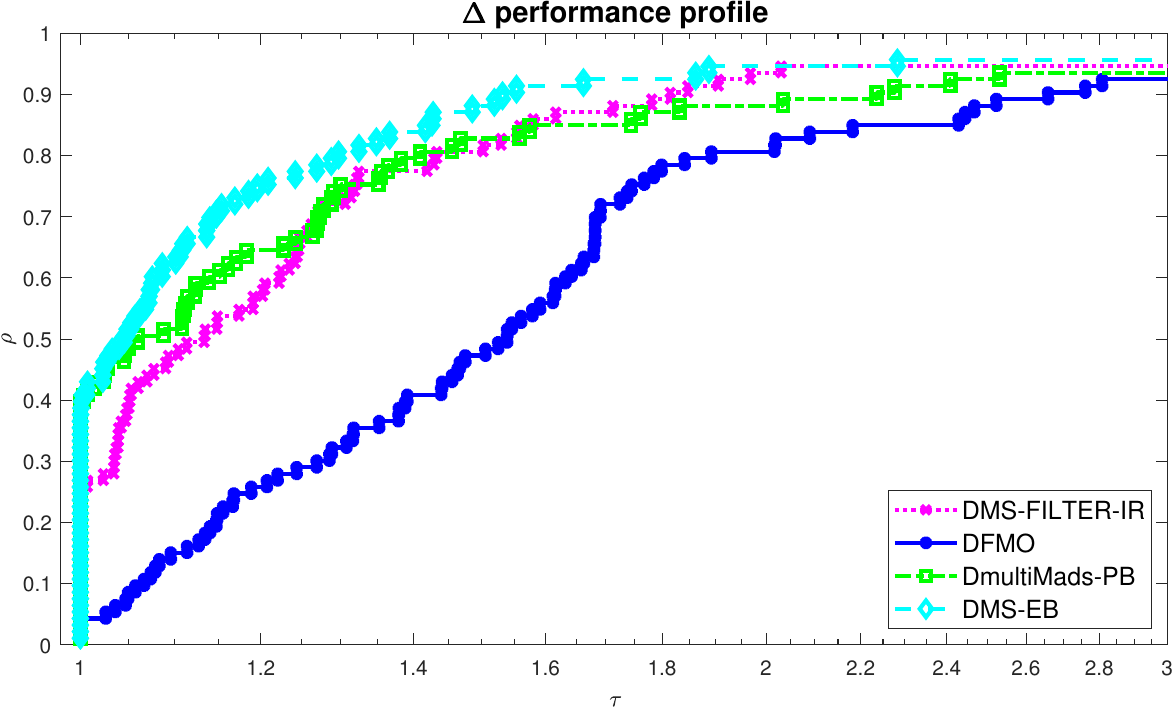}}
\caption{Comparing DMS-FILTER-IR with DFMO, DMultiMads-PB, and DMS
based on performance profiles for a maximum of $5000$ function
evaluations.} \label{performance_5000funceval}
\end{figure}

Individual comparisons between DMS-FILTER-IR and each of the three
remaining solvers considered can be found in
Figures~\ref{individualperformance_500funceval}
and~\ref{individualperformance_5000funceval}, for budgets of $500$
and $5000$ function evaluations, respectively, clarifying the
previous analysis and supporting the conclusions drawn.

\begin{sidewaysfigure}
\centering
\includegraphics[scale=0.28]{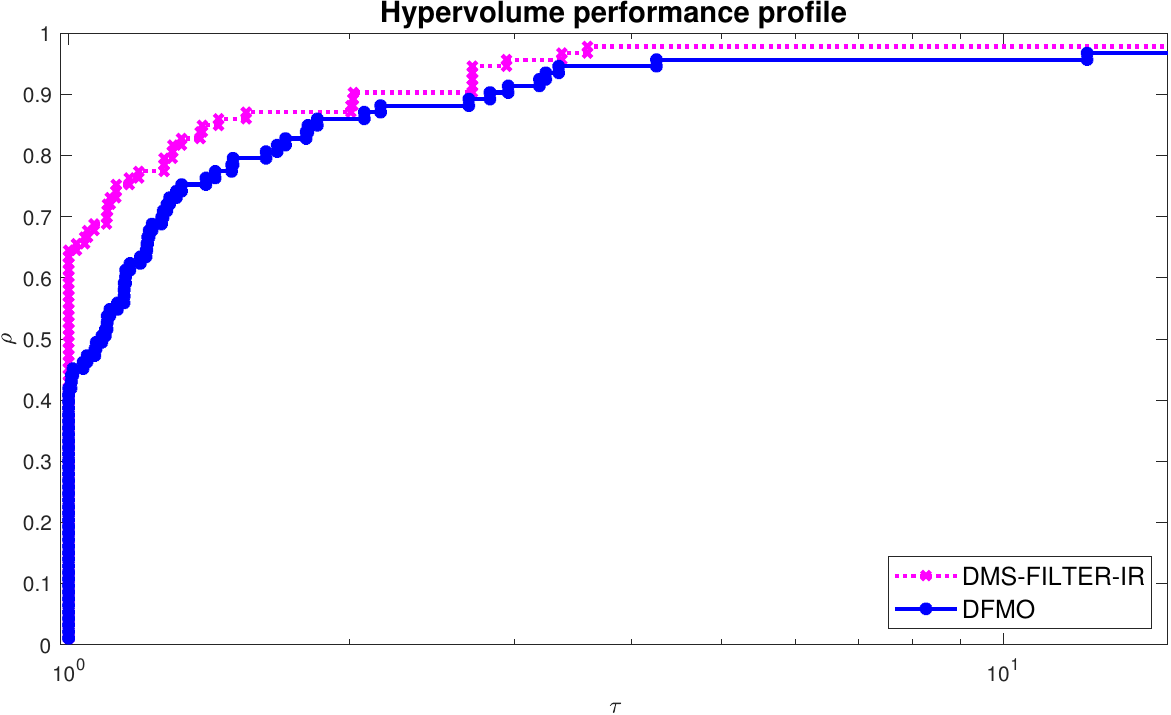}\includegraphics[scale=0.28]{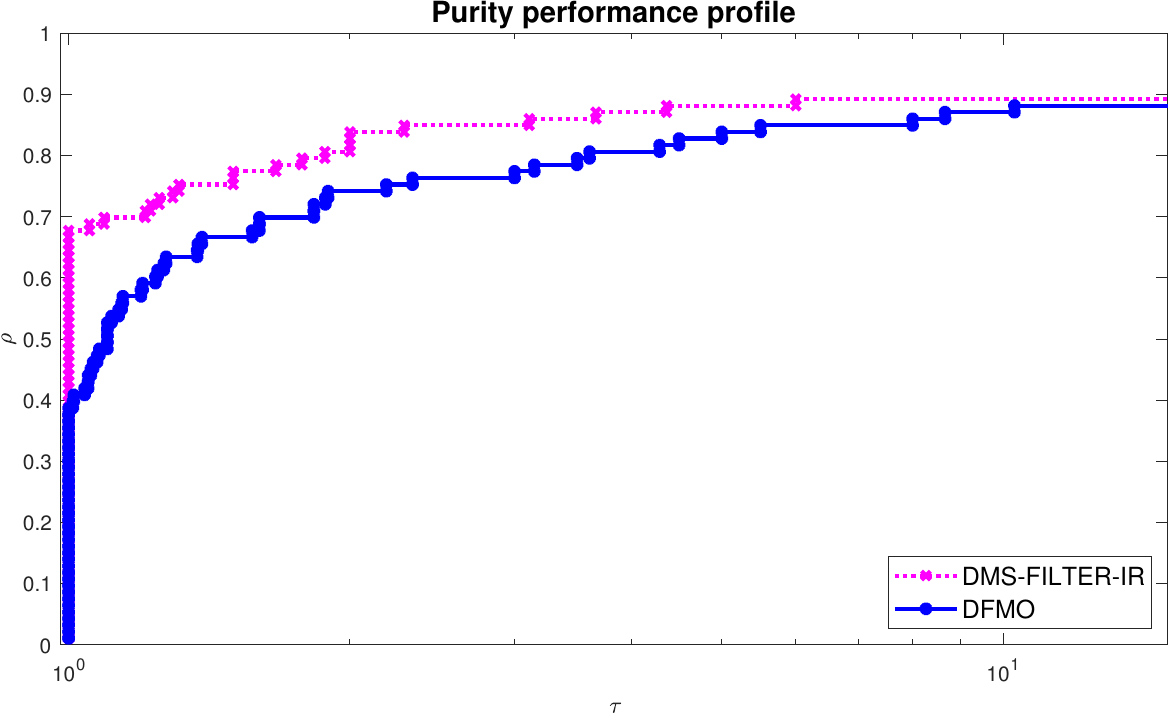}\includegraphics[scale=0.28]{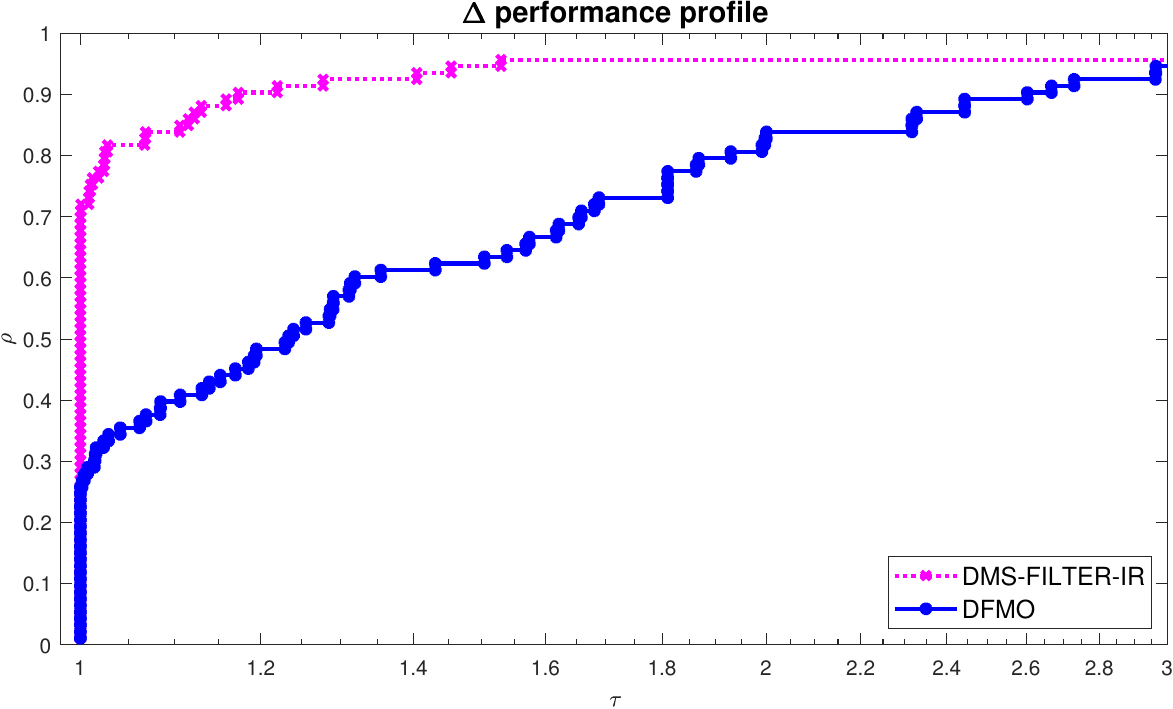}\includegraphics[scale=0.28]{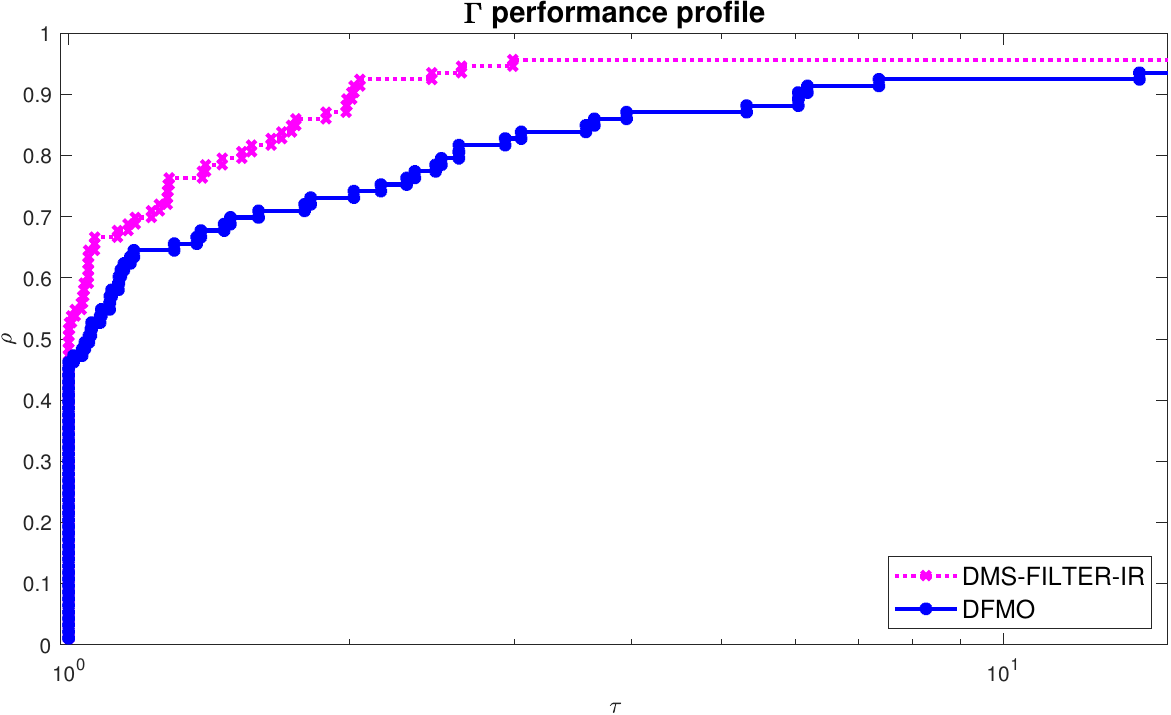}\\
\includegraphics[scale=0.28]{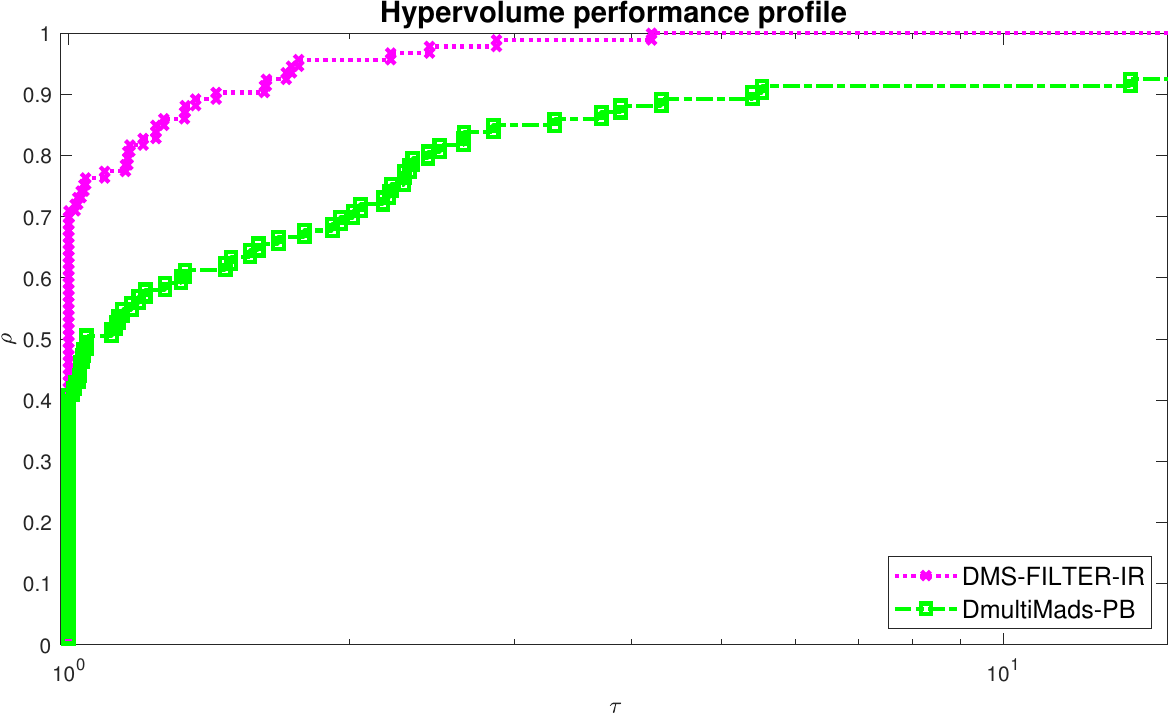}\includegraphics[scale=0.28]{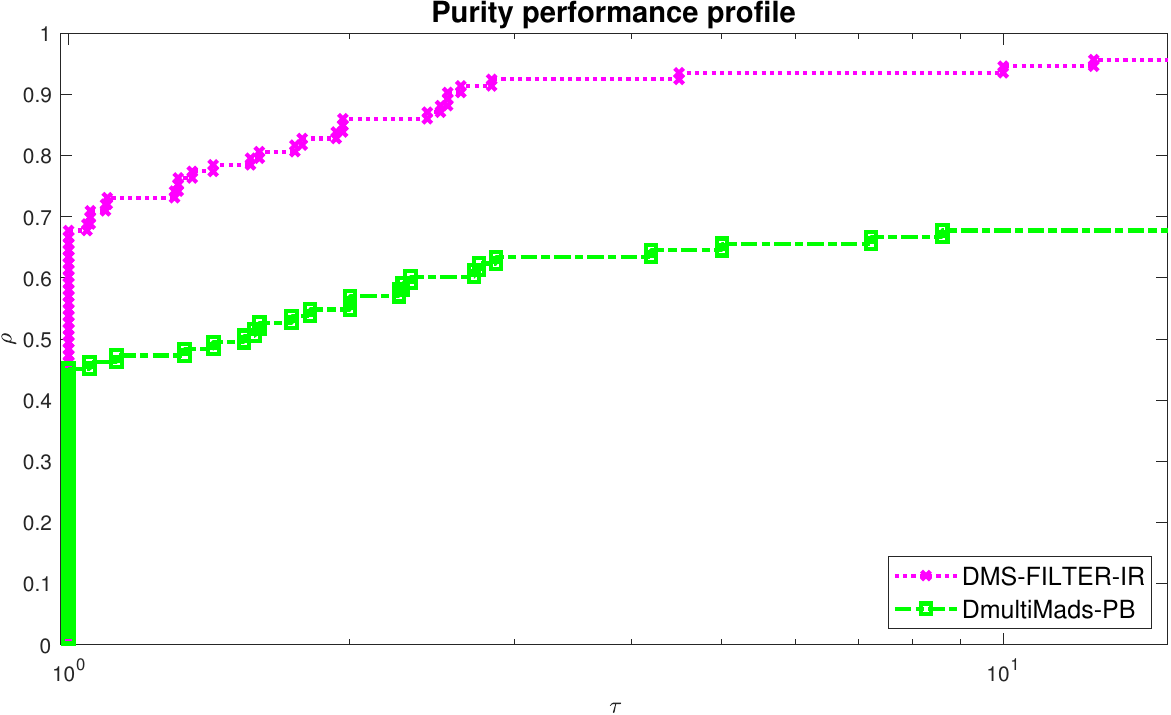}\includegraphics[scale=0.28]{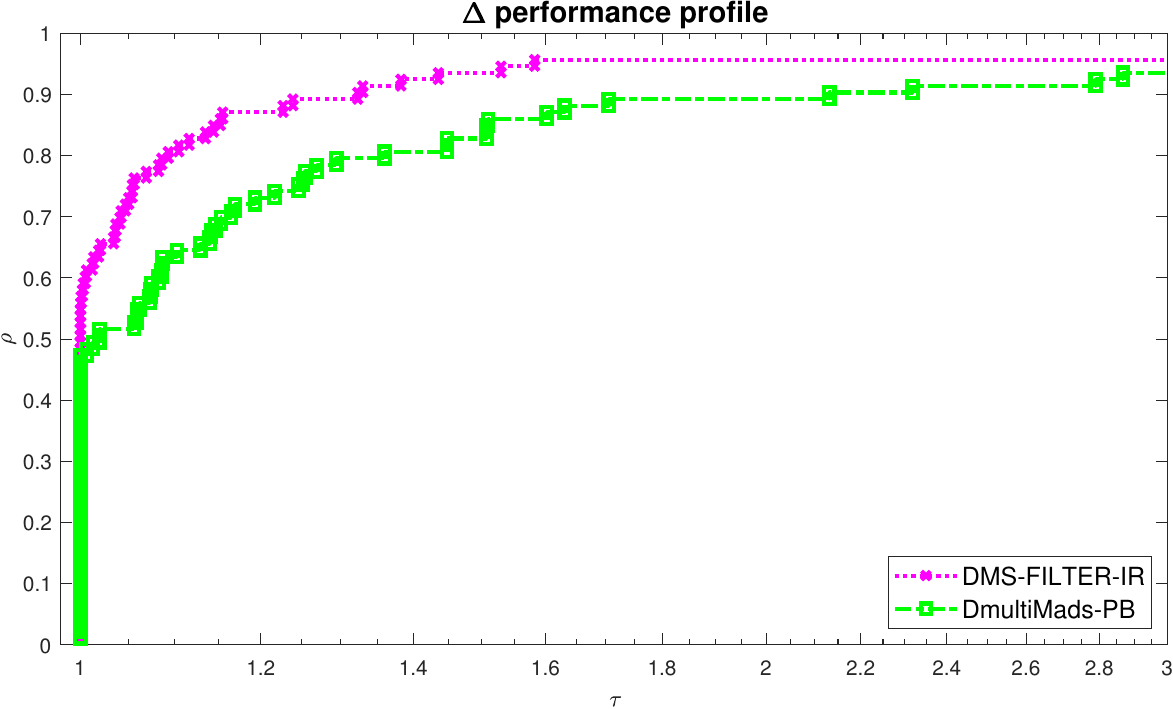}\includegraphics[scale=0.28]{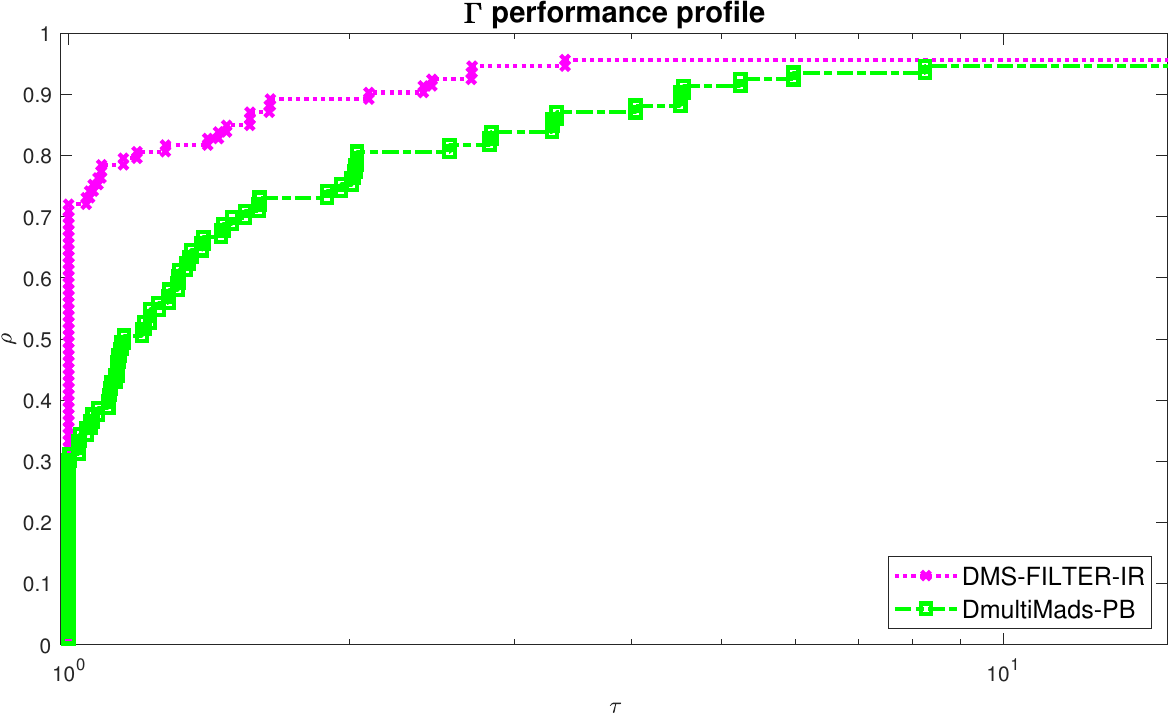}\\
\includegraphics[scale=0.28]{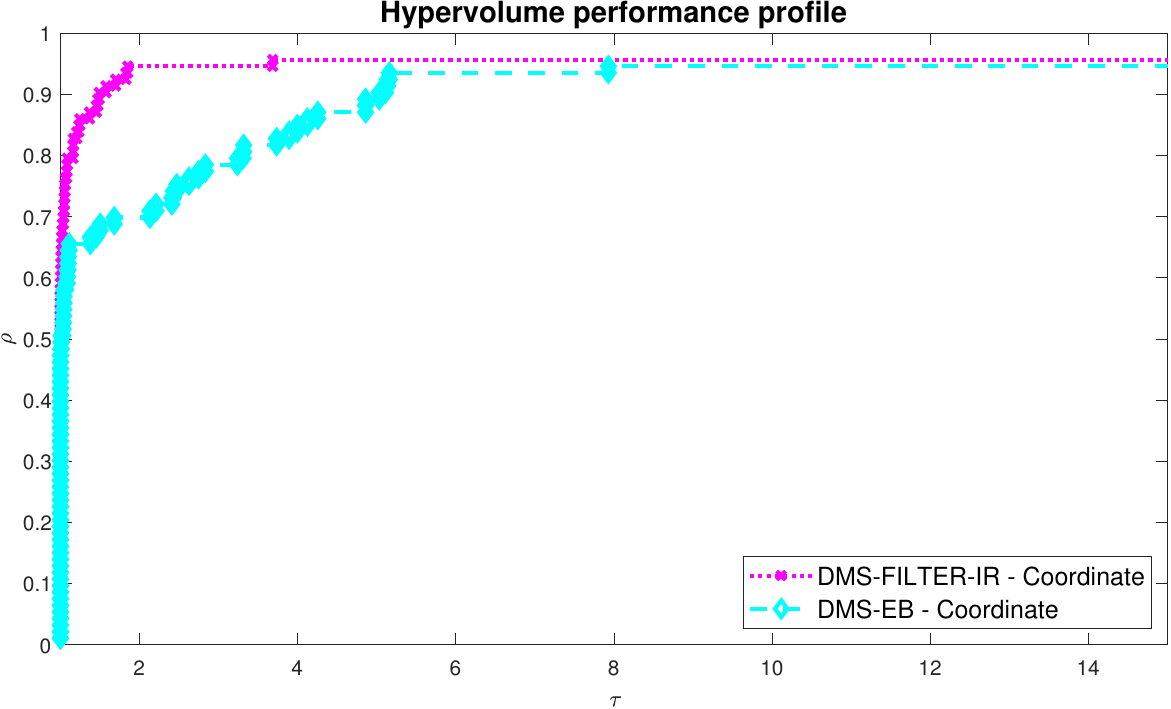}\includegraphics[scale=0.28]{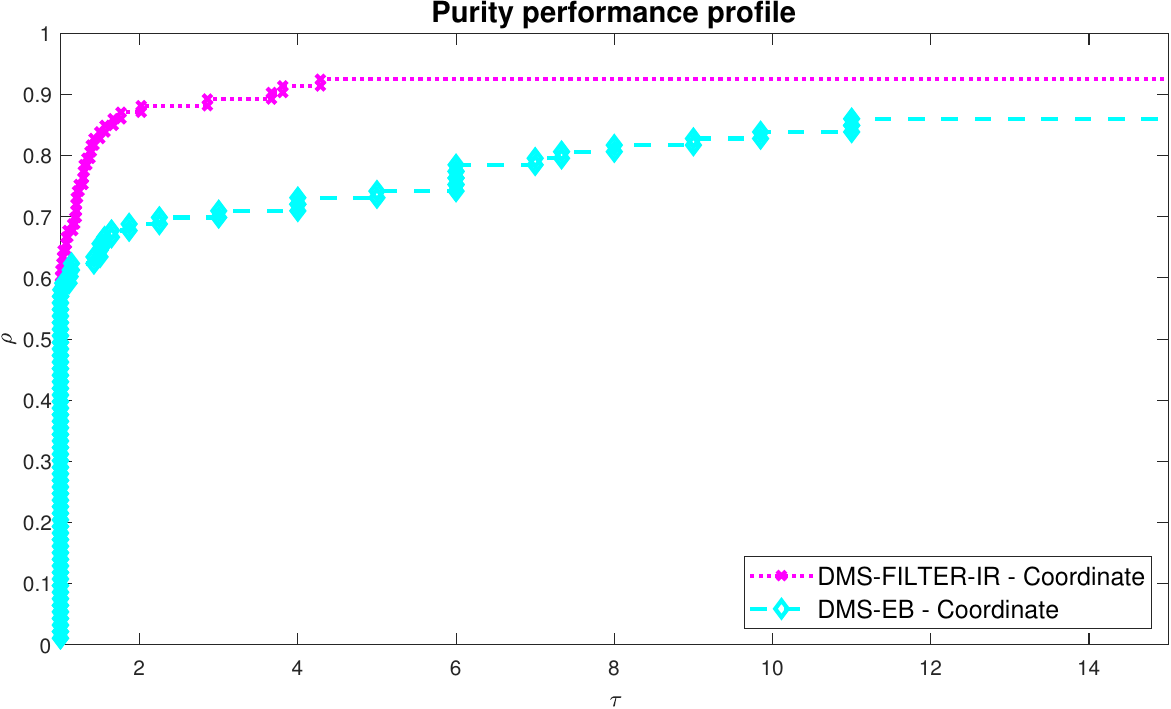}\includegraphics[scale=0.28]{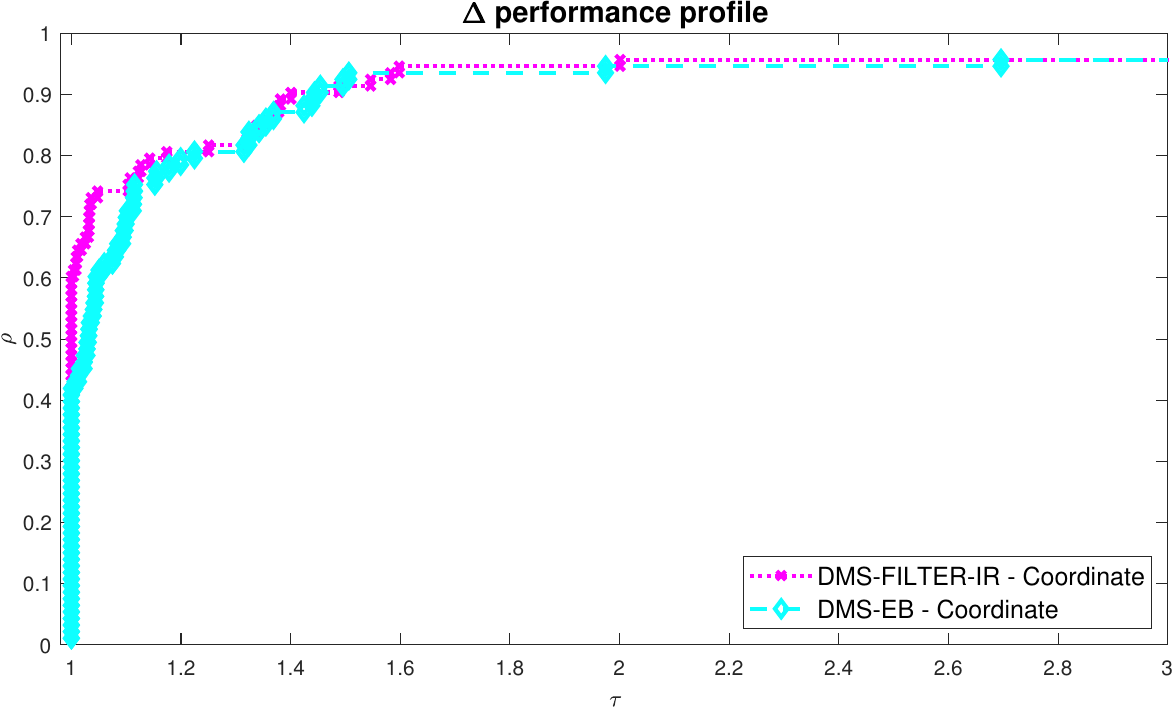}\includegraphics[scale=0.28]{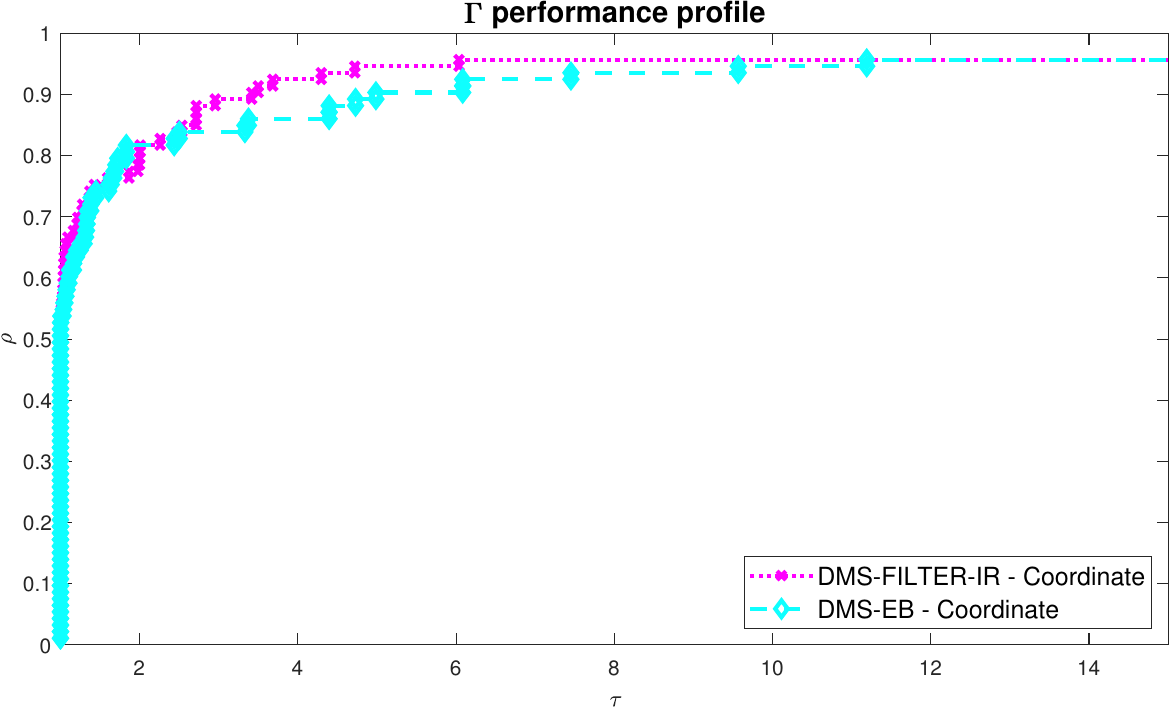}\\
\caption{Individual comparison between DMS-FILTER-IR and DFMO,
DMultiMADS-PB, and DMS based on performance profiles for a maximum
budget of $500$ function evaluations.}
\label{individualperformance_500funceval}
\end{sidewaysfigure}

\begin{sidewaysfigure}
\centering
\includegraphics[scale=0.28]{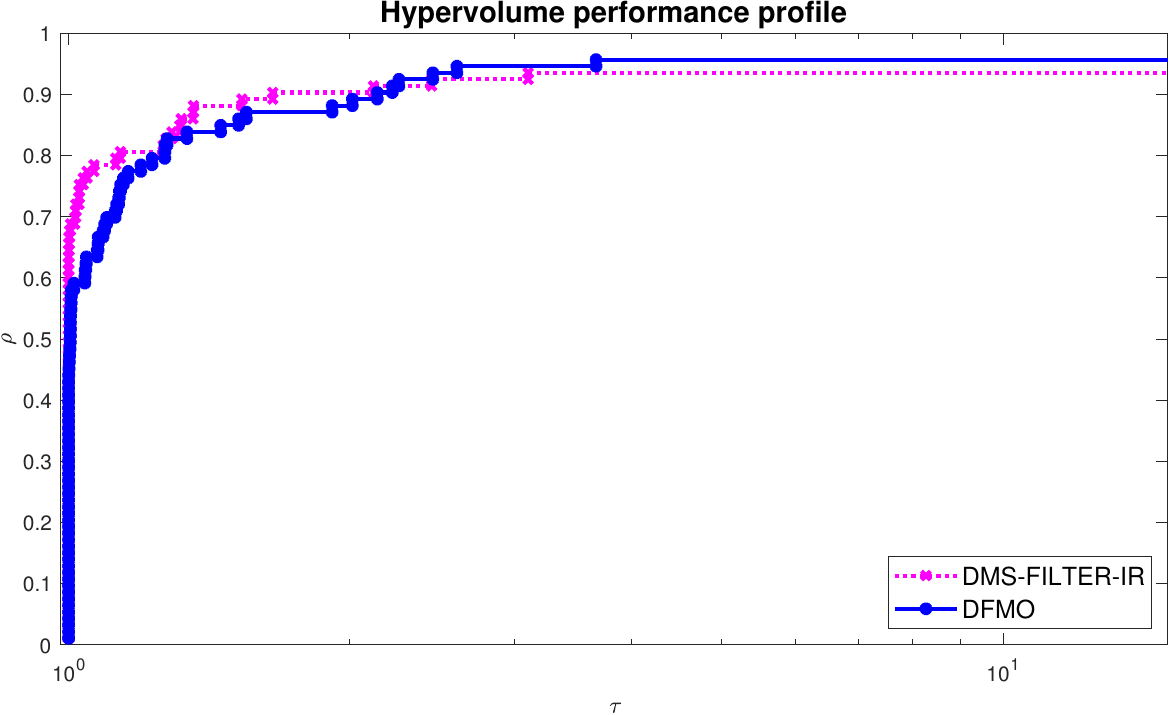}\includegraphics[scale=0.28]{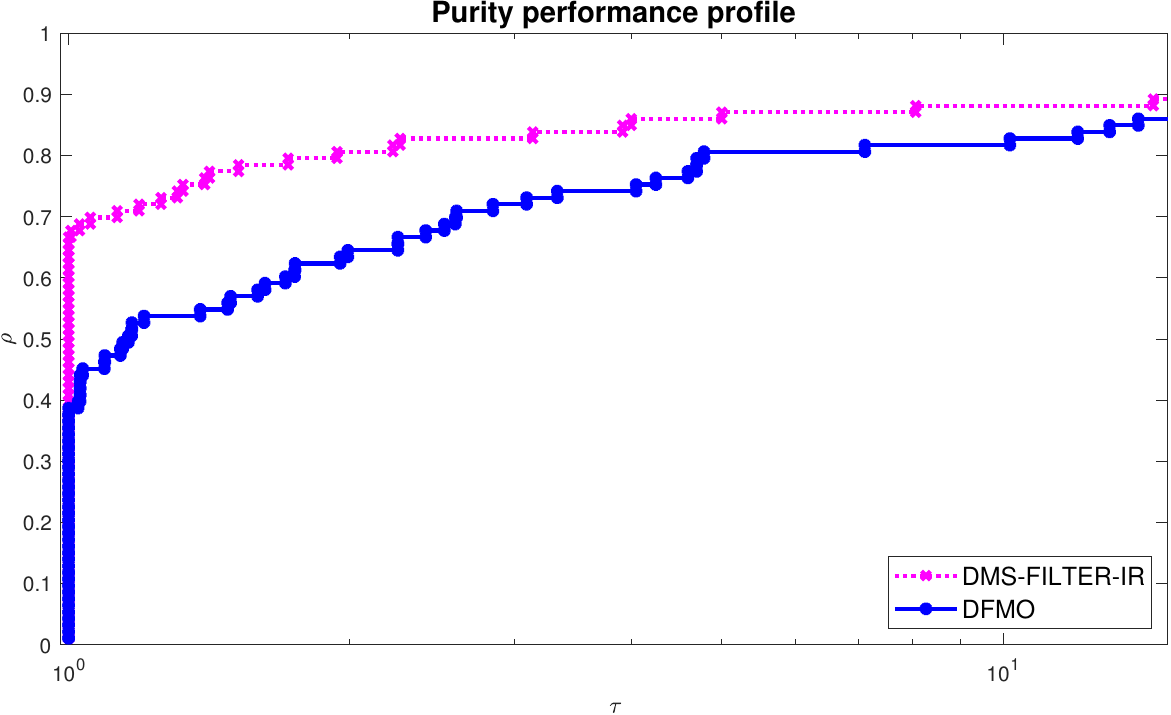}\includegraphics[scale=0.28]{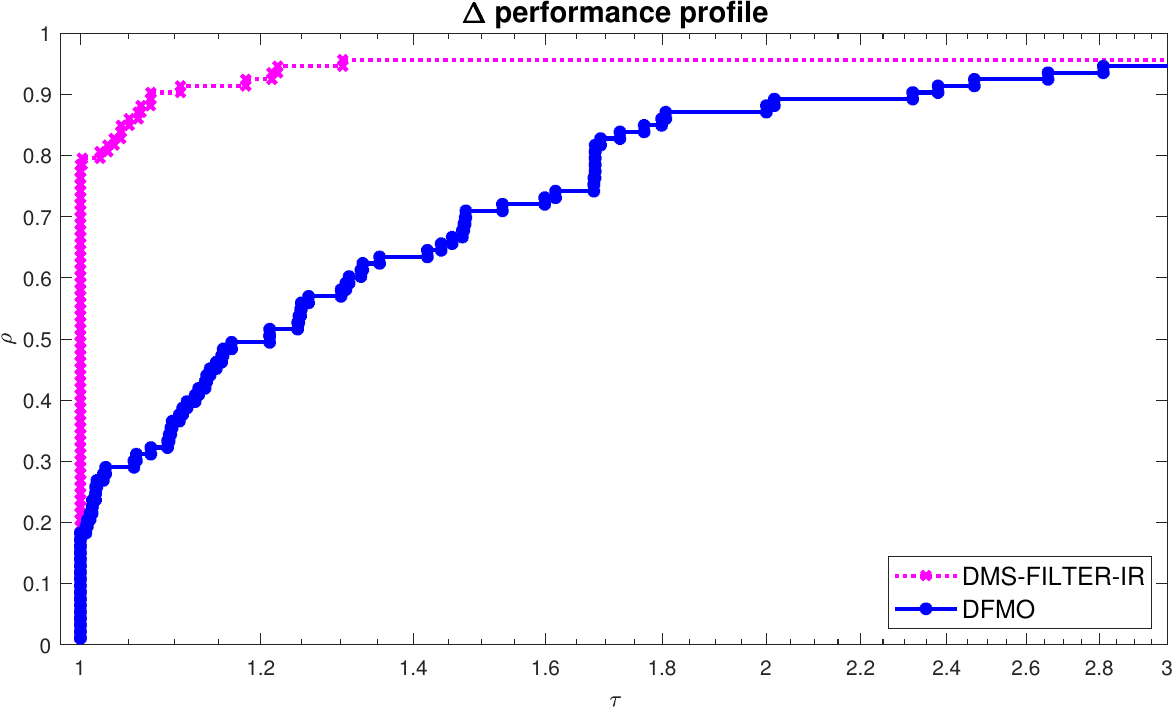}\includegraphics[scale=0.28]{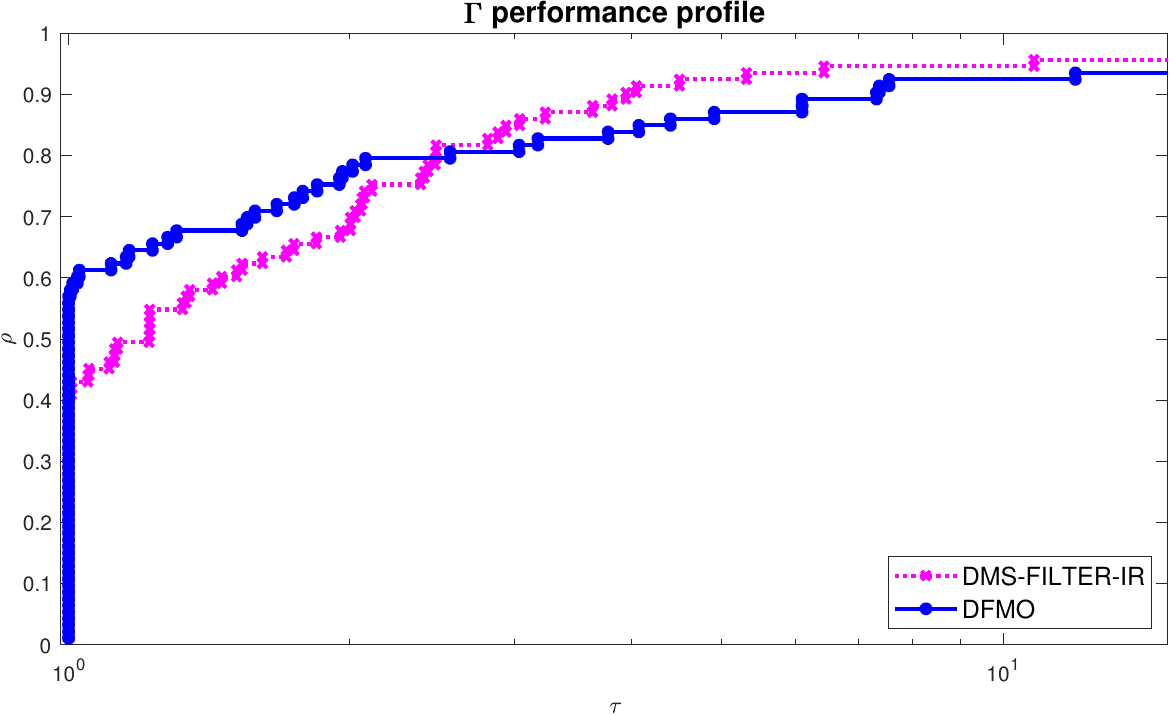}\\
\includegraphics[scale=0.28]{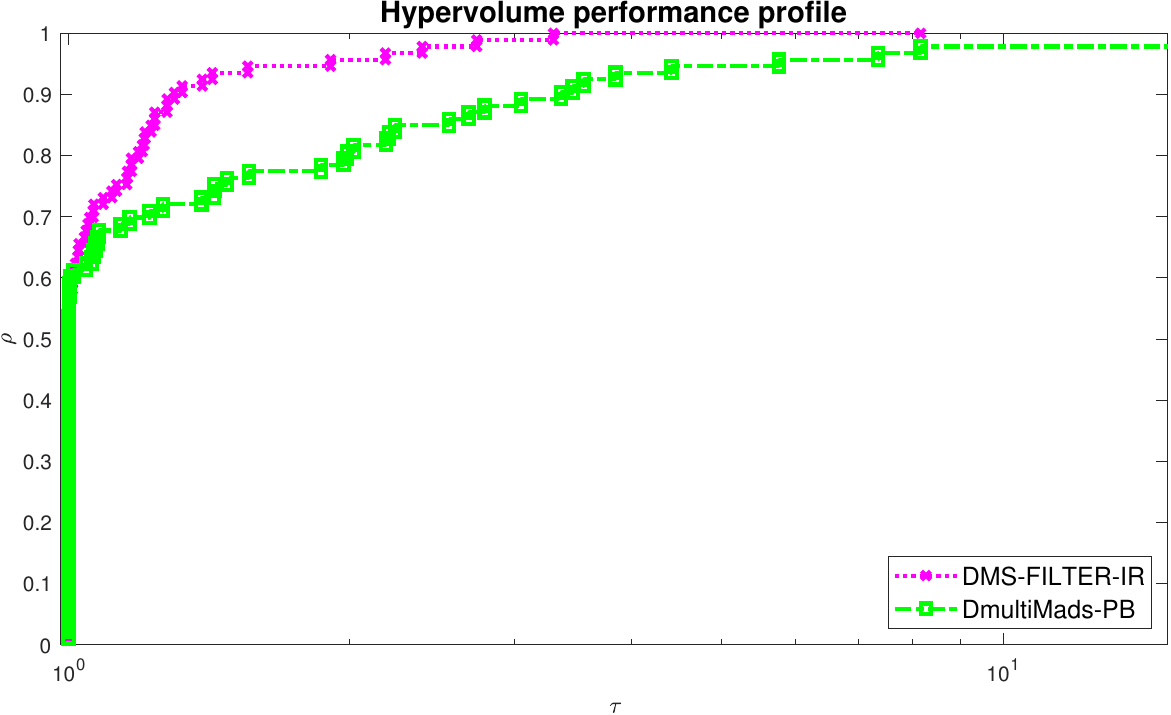}\includegraphics[scale=0.28]{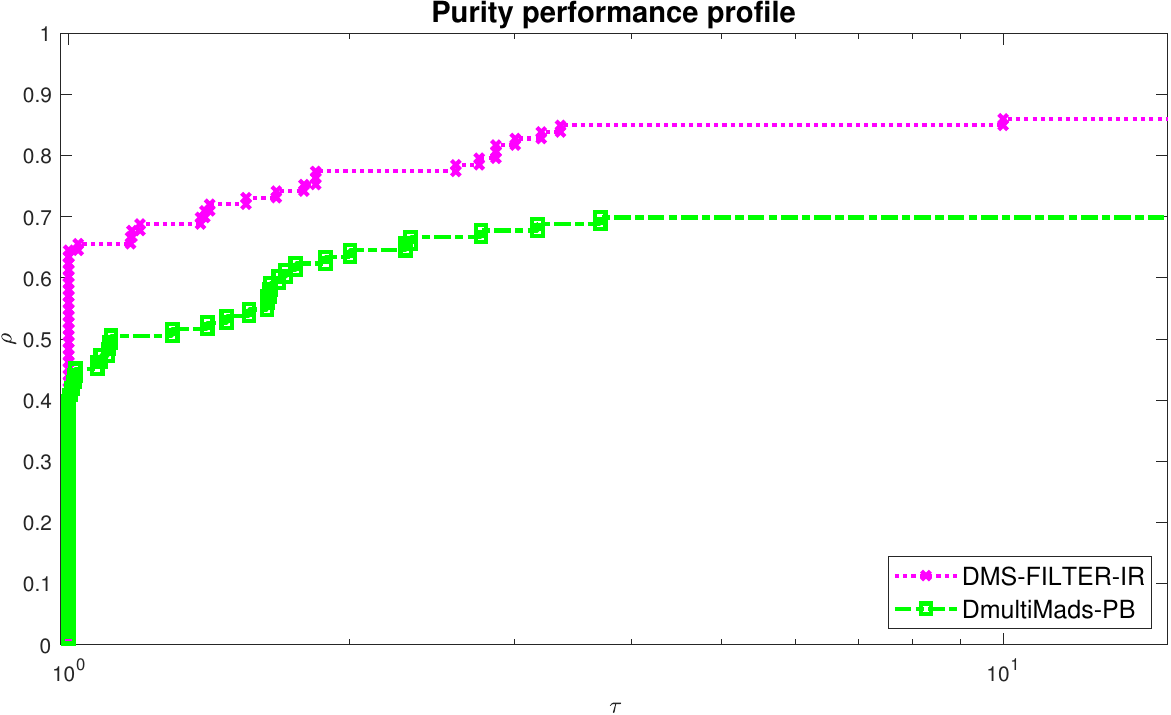}\includegraphics[scale=0.28]{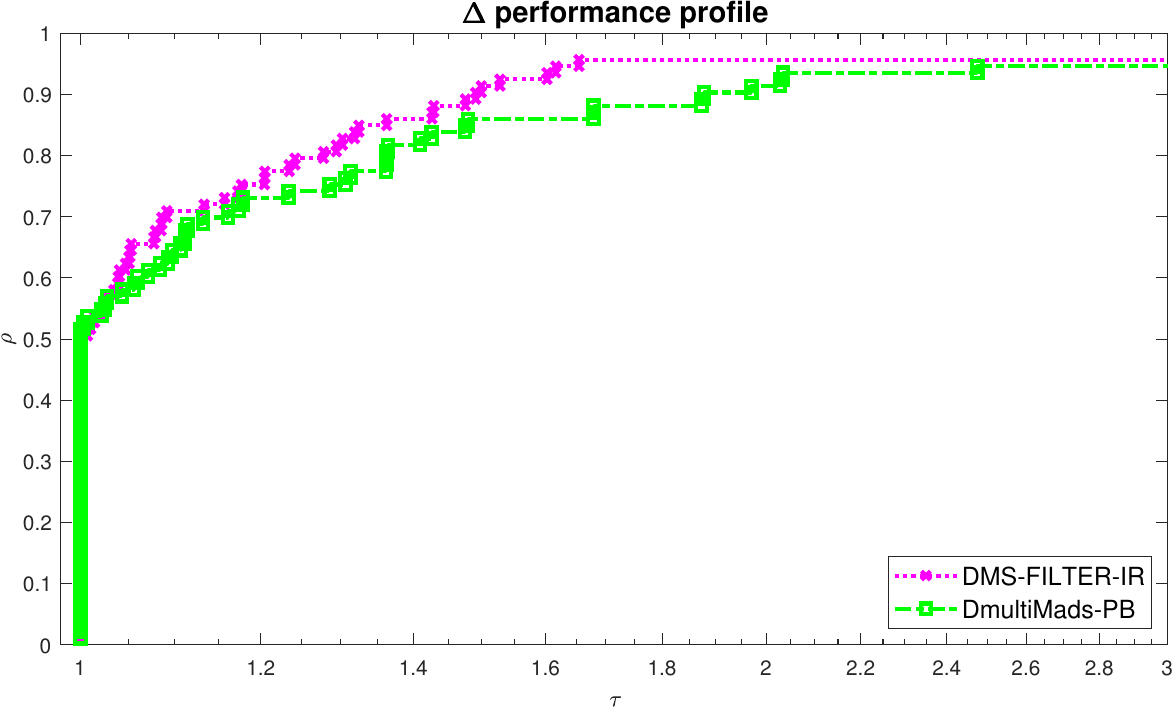}\includegraphics[scale=0.28]{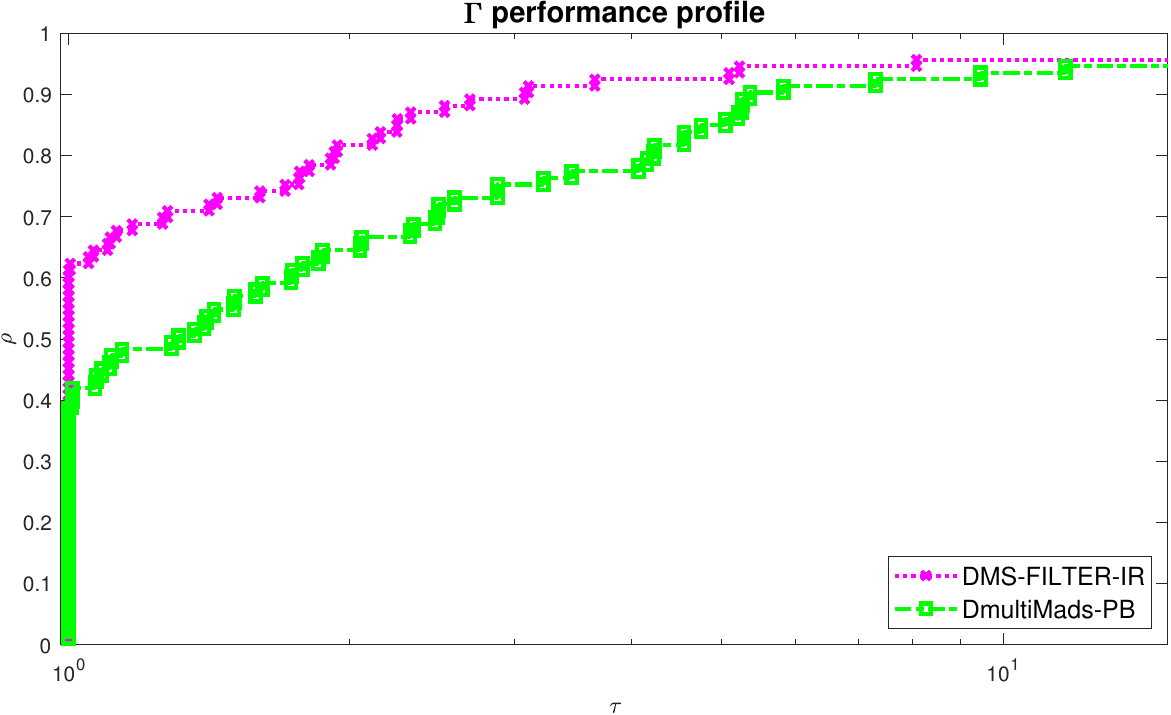}\\
\includegraphics[scale=0.28]{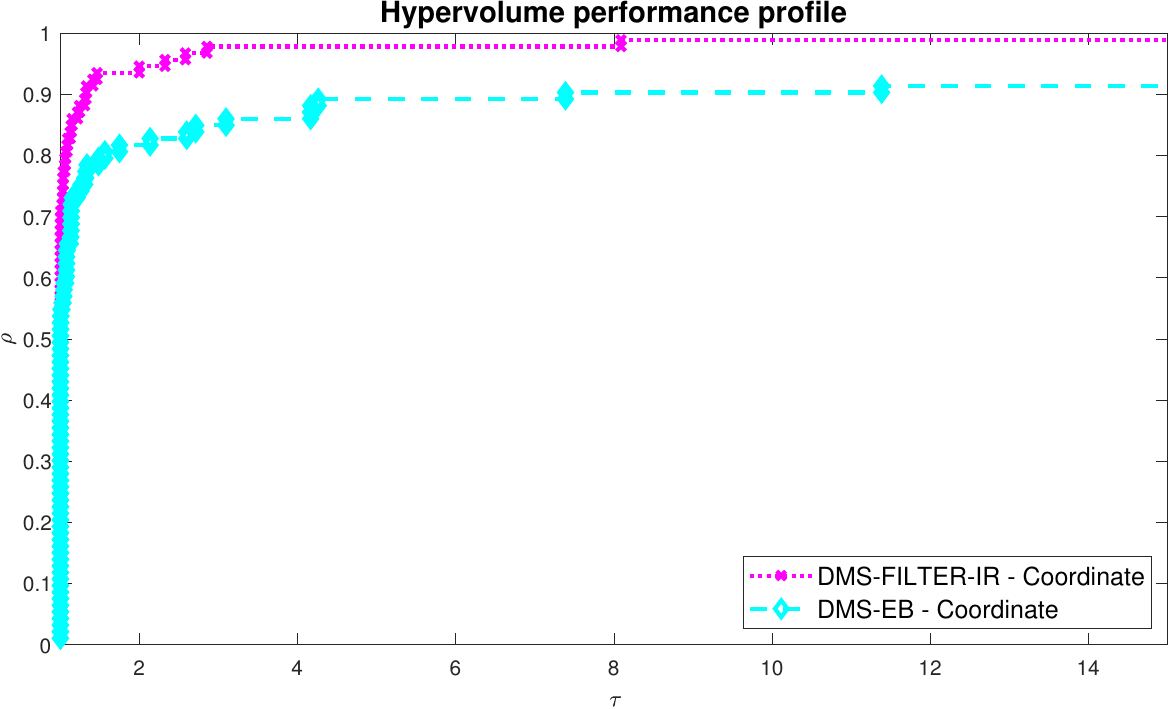}\includegraphics[scale=0.28]{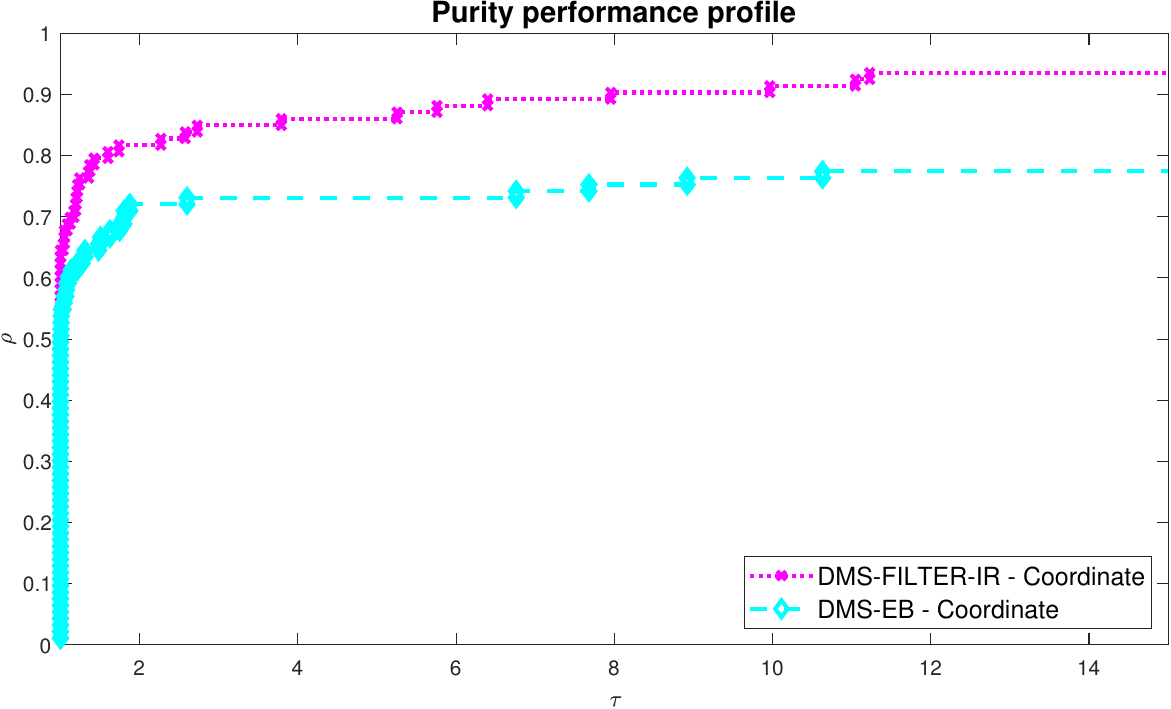}\includegraphics[scale=0.28]{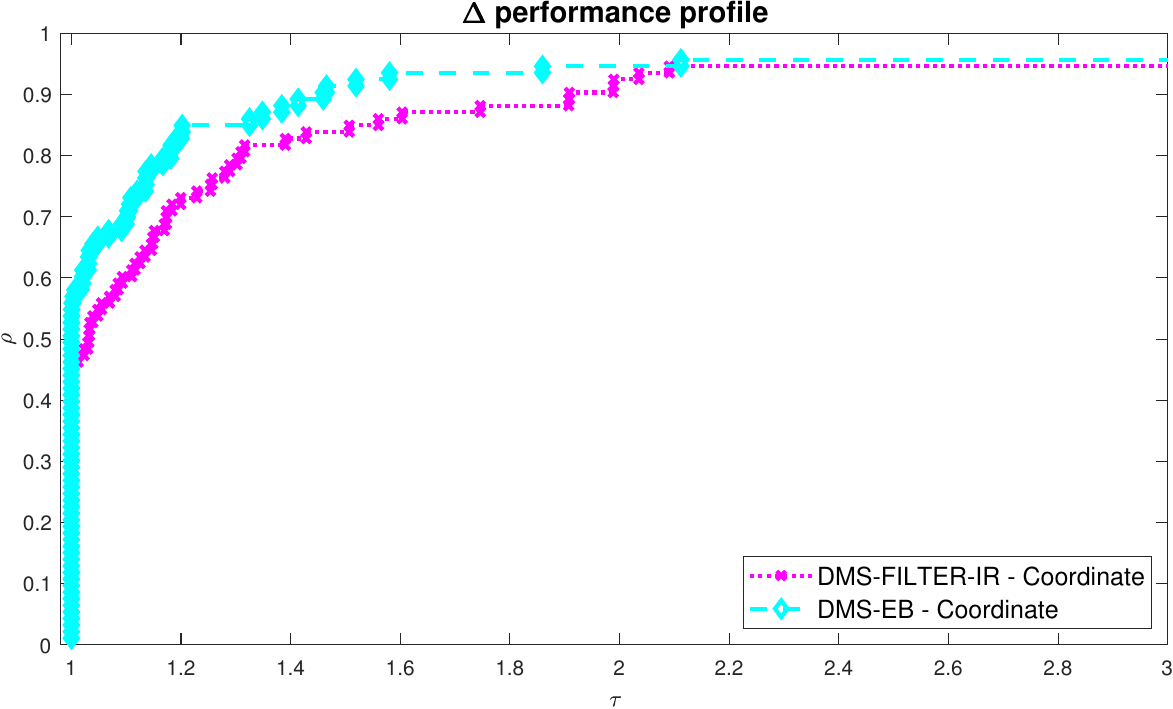}\includegraphics[scale=0.28]{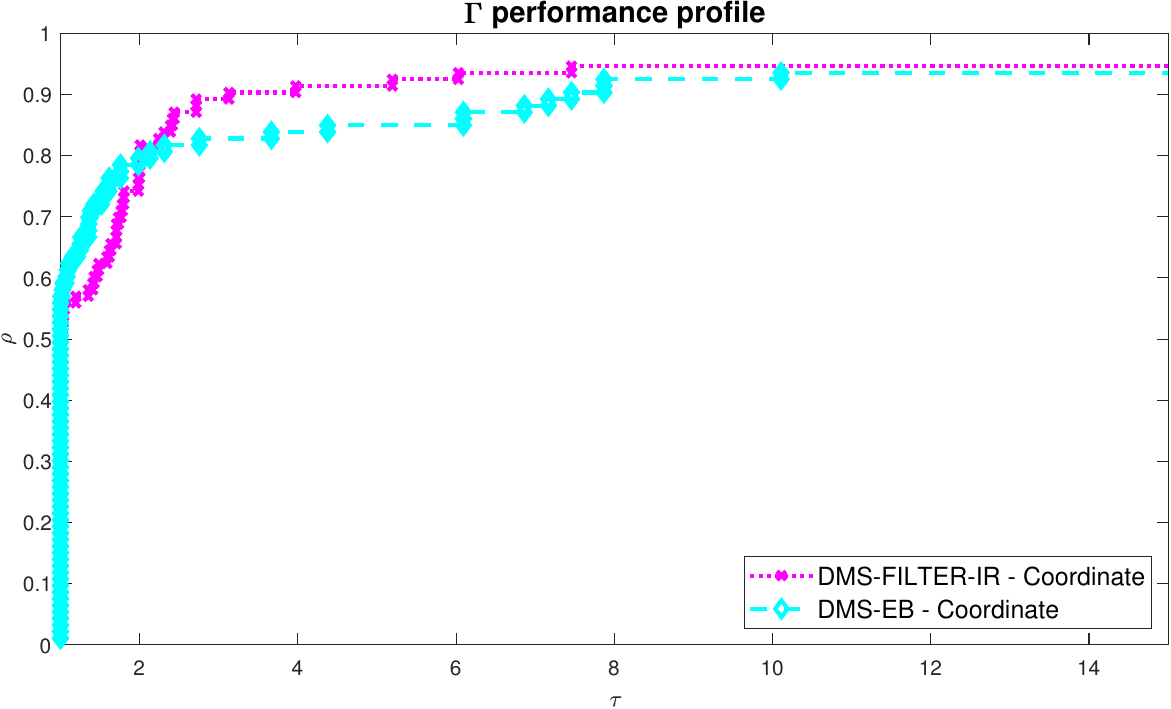}\\
\caption{Individual comparison between DMS-FILTER-IR and DFMO,
DMultiMADS-PB, and DMS based on performance profiles for a maximum
budget of $5000$ function evaluations.}
\label{individualperformance_5000funceval}
\end{sidewaysfigure}


\section{Conclusions}
\label{sec:conclusions} In this paper, we introduced, analyzed,
and tested a filter-based derivative-free approach, using an
inexact restoration step, for constrained multiobjective
optimization problems. The DMS-FILTER-IR algorithm is able to
address multiobjective derivative-free optimization problems
without a feasible initialization, regarding the relaxable
constraints.

Under the common assumptions used in derivative-free optimization
analysis, we could establish the existence of at least one
sequence of iterates generated by the algorithm that converges to
a Pareto-Clarke critical point of a related problem. Additional
conditions were provided that ensure the existence of a feasible
point and the convergence to a Pareto-Clarke critical point of the
original problem.

Extensive numerical experience allowed to establish the
competitiveness of the proposed algorithm, by comparison with
state-of-art solvers for multiobjective derivative-free
constrained optimization. Filter methods combined with inexact
restoration techniques are a valuable alternative to penalty
function methods or the use of a progressive barrier strategy.

Several extensions can be considered to improve the numerical
performance of DMS-FILTER-IR, namely the definition of a search
step, based on surrogate models~\cite{bras2020} or the use of
parallelism~\cite{tavares2022}. Moreover, techniques from
many-objective optimization literature can allow the development
of efficient numerical implementations of DMS-FILTER-IR to address
problems with more than two components in the objective function.

\bibliographystyle{abbrv}
\bibliography{refs}

\end{document}